\documentclass[a4paper,11pt]{article}
\usepackage{amsmath,amssymb,amsthm,MnSymbol}
\usepackage{paralist,enumitem,hyperref,url,bm,placeins}
\usepackage{color,graphicx}
\usepackage[margin=30mm]{geometry}

\newcommand{\R}{\mathbb{R}}
\newcommand{\ov}{\overline}
\newcommand{\bv}{\bm{v}}
\newcommand{\pd}{\partial}
\renewcommand{\div}{\, \mathrm{div} \,}
\newcommand{\I}{\mathbb{I}}
\newcommand{\bo}{\bm{\omega}}
\newcommand{\bnu}{\bm{\nu}}

\newcommand{\eps}{\varepsilon}
\newcommand{\D}{\mathrm{D}}
\newcommand{\W}{\mathrm{W}}
\newcommand{\curl}{\, \mathrm{curl}\, }
\newcommand{\bmu}{\bm{\mu}}
\newcommand{\nS}{\nabla_\Gamma}
\newcommand{\divS}{\,\mathrm{div}_\Gamma\,}
\newcommand{\pdnu}{\pd_{\bnu}}
\newcommand{\bu}{\bm{u}}
\newcommand{\nd}[1]{\pd_t^{\bullet}{#1}}
\newcommand{\inn}[2]{\langle #1, #2 \rangle}

\newcommand{\C}{\mathbb{C}}
\newcommand{\HH}{\mathbb{H}}

\newcommand{\VV}{\mathbb{V}}
\newcommand{\WW}{\mathbb{W}}
\newcommand{\UU}{\mathbb{U}}

\theoremstyle{plain}
\newtheorem{thm}{Theorem}[section]

\newtheorem{lem}[thm]{Lemma}

\newtheorem{remark}{Remark}[section]

\newtheorem{defn}{Definition}[section]
\newtheorem{assump}{Assumption}[section]
\numberwithin{equation}{section}

\title{On a phase field model for binary mixtures of micropolar fluids with non-matched densities and moving contact lines}
\author{Kin Shing Chan \footnotemark[1] \and Baoli Hao \footnotemark[2] \and Kei Fong Lam \footnotemark[1] \and Bj\"orn Stinner \footnotemark[3] } 
\date{ }

\begin{document}

\maketitle

\renewcommand{\thefootnote}{\fnsymbol{footnote}}
\footnotetext[1]{Department of Mathematics, Hong Kong Baptist University, Kowloon Tong, Hong Kong \tt(22482951@life.hkbu.edu.hk,akflam@hkbu.edu.hk)}
\footnotetext[2]{Department of Applied Mathematics, Illinois Institute of Technology, Chicago, IL \tt(bhao2@hawk.iit.edu)}
\footnotetext[3]{Department of Mathematics, University of Warwick, Coventry, U.K. \tt(bjorn.stinner@warwick.ac.uk)}

\begin{abstract}
We introduce a new phase field model for binary mixtures of incompressible micropolar fluids, which are among the simplest categories of fluids exhibiting internal rotations. The model fulfils local and global dissipation inequalities so that thermodynamic consistency is guaranteed. Our model consists of a Navier--Stokes--Cahn--Hilliard system for the fluid velocity, pressure, phase field variable and chemical potential, coupled to an additional system of Navier--Stokes type for the micro-rotation. Our model accounts for non-matched densities as well as moving contact line dynamics, and serve as a generalisation to earlier models for binary fluid flows based on a volume averaged velocity formulation. We also establish the existence of global weak solutions in three spatial dimensions for the model equipped with singular logarithmic and double obstacle potentials.

\end{abstract}



\section{Introduction}\label{sec:intro}
Conventional \emph{non-polar} fluids assume each fluid particle is not subjected to body torque forces, and via the conservation of angular momentum, the stress tensor that determines the state of stress at a given point must be symmetric. For fluids such as ferrofluids, blood flows, liquid crystals and bubbly liquids, whose particles possess an internal angular momentum separated from the external angular momentum brought about by rigid motions of the whole fluid, the well studied Navier--Stokes description from classical hydrodynamics is inadequate since these internal rotations cannot be capture without a non-symmetric stress tensor and new equations.

In a series of works which are well summarized in \cite{Eringen,Eringen1,Eringen2}, A.~Cemal Eringen proposed the concept of microcontinuum field theories to explore materials whose particles possess orientations. A hierarchy of theories starts with the class of micropolar continuum where each material point is endowed with three rigid directors, then the class of microstretch continuum where each material point have three microrotations and one microstretch, to the most general class of micromorphic continuum where each material point carries three defromable directors. The introduction of these directors in Eringen's theories necessitates new balance laws to complement the classical laws of mass, momentum and energy conservation. The conservation of micro-inertia is the missing component that complements the existing set of conservation laws to determine the evolution of the material body. 

We focus on the simplest case of isotropic, isothermal and incompressible micropolar fluids, where the equations amount to a simple but significant generalisation of the conventional Navier--Stokes equations. Let $\bu$ denote the fluid velocity, $p$ the pressure, $\eta$ the viscosity and $\bo$ the micro-rotation. Then, the balance of mass, momentum and moment of momentum for micropolar fluids reads as \cite{Eringen2,Luka}:
\begin{subequations}\label{microNS}
\begin{alignat}{2}
\div \bu & = 0, \label{MNS:div} \\
\pd_t \bu + (\bu \cdot \nabla) \bu & =  (\eta + \eta_r) \Delta \bu - \nabla p + 2 \eta_r \curl \bo, \label{MNS:mom} \\
\pd_t \bo + (\bu \cdot \nabla) \bo & = (c_0 + c_d - c_a) \nabla \div \bo + (c_a + c_d) \Delta \bo + 2 \eta_r(\curl \bu - 2 \bo), \label{MNS:w}
\end{alignat}
\end{subequations}
where $c_0$, $c_a$, $c_d$ are the coefficients of angular viscosities, and $\eta_r$ is the dynamic micro-rotation viscosity. The non-symmetry of the stress tensor can be attributed to the appearance of the term $2 \eta_r \curl \bo$ in the momentum balance \eqref{MNS:mom}. We note that setting $\eta_r = 0$ decouples the Navier--Stokes system \eqref{MNS:div}-\eqref{MNS:mom} with the balance law for $\bo$, while in the formal limit $\eta_r$ tends to infinity we obtain from \eqref{MNS:w} the relation $\bo = \frac{1}{2} \curl \bu$, which connects the micro-rotation with the well-known fluid vorticity.

As the balance law for $\bo$ is similar to the Navier--Stokes equations, many mathematical results concerning weak/strong solvability, long-time behavior, global attractors, optimal time decay rates and blow-up criteria have been established in the literature, see e.g.~\cite{Boldrini,Chen,LiuZhang,Luka,Luka2D} and the references cited therein. Of notable interest is a collect of works (of which we cite \cite{Devakar,Ramkissoon,Yadav}) that consider mixtures of incompressible micropolar and Newtonian fluids in simple geometries, such as in a long channel, which can be seen as idealized situations encountered in many industrial and medical applications \cite{Ariman}. 
 
The main focus of this work is to develop new models for binary mixtures of micropolar fluids based on the phase field approach. In the case of non-polar fluids, the basic phase field model is the Model H of Hohenberg and Halperin \cite{Hohenberg}, which couples a Cahn--Hilliard equation with a Navier--Stokes system to describe the dynamics of a binary mixture of incompressible, immiscible Newtonian fluids with the same mass densities. To generalize to the case of non-matched densities, one would have to select a mixture velocity, which can be either the mass averaged velocity \eqref{ma:velo} used e.g.~in \cite{Aki,Lowengrub,SR:Zee} or the volume averaged velocity \eqref{va:velo} used e.g.~in \cite{AGG,Boyer,Ding}. We choose the latter due to the mathematical advantage of having a solenoidal mixture velocity, which permits the model to be better amenable to further theoretical analysis compared to models based on the mass averaged velocity, see e.g.~\cite{2023CHNS} for a discussion. Our phase field model for a binary mixture of micropolar fluids reads as
\begin{subequations}\label{microNSCH}
\begin{alignat}{2}
\label{model:div} \div \bu & = 0, \\
\label{model:phi} \pd_t \phi + \bu \cdot \nabla \phi & = \div (m(\phi) \nabla \mu), \\
\label{model:mu} \mu & = \tfrac{\sigma}{\eps} F'(\phi) - \sigma \eps \Delta \phi, \\
\label{model:u} \rho(\phi) (\pd_t \bu + (\bu \cdot \nabla) \bu) & = - \nabla p - \div  ( \sigma \eps \nabla \phi \otimes \nabla \phi  ) \\
\notag & \quad  + \div (2 \eta(\phi) \D \bu + 2 \eta_r(\phi) \W \bu)  \\ 
\notag & \quad + 2 \curl(\eta_r(\phi) \bo) + \tfrac{\overline{\rho}_1 - \overline{\rho}_2}{2} \big ( m(\phi) \nabla \mu \cdot \nabla \big ) \bu, \\
\label{model:w} \rho(\phi)(\pd_t \bo + (\bu \cdot \nabla) \bo) & = \div (c_0(\phi) (\div \bo) \I + 2c_d(\phi) \D \bo +2 c_a(\phi) \W \bo)  \\
\notag  & \quad + 2 \eta_r(\phi)(\curl \bu - 2 \bo) +  \tfrac{\overline{\rho}_1 - \overline{\rho}_2}{2} \big ( m(\phi) \nabla \mu \cdot \nabla \big ) \bo, 
\end{alignat}
\end{subequations}
which is an extension to the model of Abels, Garcke and Gr\"un \cite{AGG} to the case of micropolar fluids. In the above $\bu$ is now the volume averaged velocity, $\phi$ is the phase field variable, $\mu$ is the associated chemical potential, $F'$ is the derivative of a double well function $F$ possessing two equal minima, $\sigma$ is the surface tension, $\eps$ is a small parameter related to the interfacial thickness, $m(\phi)$ is a mobility function and $\overline{\rho}_i$, for $i = 1, 2$, is the constant mass densities of the $i$th fluid. The phase field variable typically takes values in between the two equal minima of $F$. For example the classical quartic potential $F(s) = \frac{1}{4}(s^2-1)^2$ has two equal minima at $s = \pm 1$, and we can assign the region occupied by fluid 1 as the set $\{\phi = 1\}$, the region occupied by fluid 2 as $\{\phi = -1\}$ and the interfacial layer separating these regions as $\{ |\phi| < 1 \}$. The mass density $\rho$ is now a function of $\phi$ and is given by the algebraic relation, see also \eqref{hat:rho:def}:
\begin{align}\label{intro:rho}
\rho(\phi) = \frac{\overline{\rho}_1 - \overline{\rho}_2}{2} \phi + \frac{\overline{\rho}_1 + \overline{\rho}_2}{2} \quad \text{ for } \phi \in [-1,1].
\end{align}
Furthermore, $\D \bu = \frac{1}{2}(\nabla \bu + (\nabla \bu)^{\top})$ is the symmetric strain tensor, while $\W \bu = \frac{1}{2}(\nabla \bu - (\nabla \bu)^{\top})$ is the anti-symmetric strain tensor. Notice that the presence of $\W \bu$ means the stress tensor is no longer symmetric.

We note the resemblance of \eqref{microNSCH} to the single component model \eqref{microNS}, whereby setting $\phi$ and $\mu$ to be constant, setting $\overline{\rho}_1 = \overline{\rho}_2$, and neglecting the Cahn--Hilliard submodel \eqref{model:phi}-\eqref{model:mu} leads to \eqref{microNS}. For a binary mixtures we allow both fluids to have possibly different viscosity coefficients $\{\eta_i, \eta_{r,i}, c_{0,i}, c_{d,i}, c_{a,i}\}_{i=1,2}$, and hence in \eqref{microNSCH} the functions $\eta(\phi)$, $\eta_r(\phi)$, $c_0(\phi)$, $c_d(\phi)$ and $c_a(\phi)$ serve to interpolate between the two different sets of fluid viscosities in a smooth fashion. In \eqref{model:u}, the term $\div (\sigma \eps \nabla \phi\otimes \nabla \phi)$ in \eqref{model:u} accounts for capillary forces due to surface tension, while $\tfrac{\overline{\rho}_1 - \overline{\rho}_2}{2} \big ( m(\phi) \nabla \mu \cdot \nabla \big ) \bu$ describes transport by the relative flux $\bm{J} = -\tfrac{\overline{\rho}_1 - \overline{\rho}_2}{2}  m(\phi) \nabla \mu$ related to the diffusion of the components \cite{AGG}, which vanishes in the case of matched densities $\overline{\rho}_1 = \overline{\rho}_2$. 

To complete the model \eqref{microNS} we furnish it with appropriate boundary conditions. The classical ones are the homogeneous Neumann conditions for $\phi$ and $\mu$, and homogeneous Dirichlet boundary conditions for $\bu$ and $\bo$. The Neumann condition for $\mu$ provides the conservation of mass expressed in the form of the mean value of $\phi$, while the Neumann condition for $\phi$ describes a static contact of angle of $\pi/2$ between the domain boundary $\pd \Omega$ and the fluid-fluid interface.  Without micro-rotation, i.e., $\bo = \bm{0}$, the resulting Navier--Stokes--Cahn--Hilliard model has been studied extensively by many authors. We refer to \cite{ADG,ADG2} for global weak solutions, to \cite{AG1,AG2} for local strong solutions, \cite{2023CHNS,AG4} for regularity properties, to \cite{CCGMG1,CCGMG2,CCGMG3} for long time behavior in the context of attractors, and to \cite{AbelsTera,Frigeri1,Frigeri2,CCGAGMG} for similar results concerning a non-local variant of the Navier--Stokes--Cahn--Hilliard system, see Section~\ref{sec:nonlocal} below.  

On the other hand, for situations where the interaction between the fluid-fluid interface and the domain boundary is significant, the homogeneous Neumann condition for $\phi$ is too restrictive and the dynamics occuring near the domain boundary can be inconsistent with the no-slip boundary condition for $\bu$. The contact line, as the intersection of the fluid-fluid interface and the domain boundary, may be mobile and the contact angles between the contact line and the domain boundary may vary in time. Various remedies have been proposed, see e.g.~\cite{Mhei,Djac2,QianPhys,Qian2,Qian,Psepp,WQR1,WQR2,WQR3} and the references cited therein, and for our model \eqref{microNSCH} we consider, as in \cite{QianPhys,Qian2}, a generalized Navier boundary condition (GNBC) for $\bu$, together with a dynamic contact line condition for $\phi$, which can be derived via a Coleman--Noll procedure. This reads as
\begin{subequations}\label{micro:bc}
\begin{alignat}{2}
\label{model:bc:1} & \pdnu \mu = 0, \quad \bu \cdot \bm{\nu} = 0, \quad \bo \cdot \bm{\nu} = 0, \\
\label{model:bc:2} & \pd_t \phi + \bu_\tau \cdot \nS \phi = - \mathcal{L}(\phi), \\
\label{model:bc:3} & \mathcal{L}(\phi) = - \Delta_\Gamma \phi + \pdnu \phi + \hat{G}'(\phi), \\
\label{model:bc:4} & \bu_\tau + \Big ( 2 \bm{\nu} \cdot [ \eta(\phi) \D \bu + \eta_r(\phi) \W \bu]  + 2 \eta_r(\phi) \bo \times \bm{\nu} \Big )_\tau = \mathcal{L}(\phi) \nS \phi, \\
\label{model:bc:5} & \bo_\tau + \Big ( 2 \bm{\nu} \cdot [ c_d(\phi) \D \bo + c_a(\phi) \W \bo] \Big )_\tau = \bm{0},
\end{alignat}
\end{subequations}
where $\bm{\nu}$ is the external unit normal on the boundary $\pd \Omega = \Gamma$, $\pdnu f= \nabla f \cdot \bm{\nu}$ is the normal derivative, $\nS$ is the surface gradient operator, $\Delta_\Gamma$ is the Laplace--Beltrami operator, $\bu_\tau := \bu - (\bu \cdot \bm{\nu}) \bm{\nu}$ is the tangential component of $\bu$, and $\mathcal{L}(\phi)$ can be interpreted as the boundary chemical potential associated to an interfacial free energy density of Ginzburg--Landau type, see \eqref{chi:defn}. Equation \eqref{model:bc:1} summarizes the no-flux condition for $\mu$ and the no-penetration conditions for $\bu$ and $\bo$, while \eqref{model:bc:2} is a dynamic boundary condition of Allen--Cahn type (see the recent works \cite{GLW,AG3,KnopfStange} for an alternative Cahn--Hilliard type). The generalized Navier boundary condition (GNBC) with the uncompensated Young stress for $\bu$ is formulated as \eqref{model:bc:4} with additional contributions involving $\W \bu$ and $\bo \times \bm{\nu}$ due to micropolar effects, while \eqref{model:bc:5} is a frictionless boundary condition with relaxation for $\bo$, cf.~\cite{Ebenbeck}. 

Our aim is to analyze the new phase field model \eqref{microNSCH} subjected to the boundary conditions \eqref{micro:bc} and appropriate initial conditions. In the case of matched densities $\overline{\rho}_1 = \overline{\rho}_2$ and absence of micropolar effects $\bo = \bm{0}$, the resulting Navier--Stokes--Cahn--Hilliard model with GNBC has been studied in \cite{GGM} where the existence of global weak solutions and convergence to equilibrium are established. This is later generalised to the non-matched density case in \cite{GGW} where the existence of global weak solutions are shown when the Cahn--Hilliard submodel is furnished with a singular potential $F$ of logarithmic type:
\begin{align}\label{intro:log}
F(\phi) = \frac{\theta}{2}((1+\phi) \ln(1+\phi) + (1-\phi) \ln(1-\phi)) - \frac{\theta_c}{2} \phi^2,
\end{align}
with constants $0 < \theta < \theta_c$, in order to ensure the phase field variable $\phi$ stays in the interval $[-1,1]$, see also the recent contribution \cite{GLW} where the dynamic boundary condition \eqref{model:bc:2}-\eqref{model:bc:3} is replaced with a Cahn--Hilliard type system. The requirement for $\phi$ to stay within the physically relevant interval $[-1,1]$ is primarily due to the fact that the algebraic expression \eqref{intro:rho} for the mass density may not make sense for $\phi \notin [-1,1]$ and may not be a priori bounded from below by a positive constant. Hence, the aforementioned references on the analytical study of the Abels--Garcke--Gr\"un model often employ the logarithmic potential \eqref{intro:log} or the double obstacle potential:
\begin{align}\label{intro:obs}
F(\phi) = \begin{cases}
\frac{1}{2}(1-\phi^2) & \text{ if } |\phi| \leq 1, \\
+\infty & \text{ otherwise}.
\end{cases}
\end{align}
The strategy for establishing the existence of global weak solutions to \eqref{microNSCH}-\eqref{micro:bc} follows along a similar procedure to that of \cite{GGW,GLW}. We first study an approximate problem that is obtained by extending the definition of the mass density function $\rho$ from $[-1,1]$ to the whole of $\R$ in a fashion that maintains its positivity and boundedness, approximating the singular potential $F$ by a sequence of smooth polynomial functions, and introducing regularisation terms in the form of $\pd_t \phi$ in \eqref{model:mu}, non-Newtonian stress-like terms $-\div (|\nabla \bu|^{q-2} \nabla \bu) + |\bu|^{q-2} \bu$ and an extra term of the form $\frac{R}{2}\bu$ to \eqref{model:u} where $R = -\frac{m(\phi)}{2} [\nabla \rho'(\phi) \cdot \nabla \mu]$, and similarly to \eqref{model:w}. The extra terms ensure that a dissipation energy inequality for this approximation problem is preserved as we extend $\rho$ from $[-1,1]$ to $\R$. Existence of weak solutions to the approximate problem is demonstrated via an implicit time discretisation combined with the Leray--Schauder principle, where we have to extend the arguments in previous works to account for the coupling with the micropolar equation \eqref{model:w}. Then, passing to the limit first to obtain the singular potential $F$ ensures $\phi \in [-1,1]$ and $R = 0$, and later in another limiting process we remove the other regularisation terms to recover the original system.

The structure of the paper is as follows: In Section \ref{sec:md} we derive the new micropolar model \eqref{microNSCH}-\eqref{micro:bc} from conservation laws and demonstrate thermodynamical consistency via local and global dissipation inequalities. In Section \ref{sec:analysis} we establish the existence of global weak solutions to \eqref{microNSCH}-\eqref{micro:bc} in three spatial dimensions where $F$ is a singular potential of logarithmic type. In Section \ref{sec:obstacle} we perform a deep quench limit $\theta \to 0$ in \eqref{intro:log} to obtain global weak solutions to \eqref{microNSCH}-\eqref{micro:bc} where $F$ is the double obstacle potential \eqref{intro:obs}.

\section{Model derivation} \label{sec:md}
We consider an open bounded domain $\Omega \subset \R^3$ containing a mixture of two immiscible, incompressible, isothermal and isotropic micropolar fluids with possibly different mass densities. Let $V$ denote the total volume of the mixture and $V_i$, $i \in \{1,2\}$ the volume of the $i$th fluid. Then, the \emph{volume fraction} $\phi_i$ of the $i$th fluid is the ratio $\phi_i := V_i/V$, so that $\phi_1 + \phi_2 = 1$.  Similarly, let $M$ denote the total mass of the mixture and $M_i$, $i \in \{1,2\}$ the mass of the $i$th fluid. Then, the \emph{mass concentration} $c_i$ of the $i$th fluid is the ratio $c_i = M_i/M$, so that $c_1 + c_2 = 1$.

Since we assumed each fluid is incompressible, we can prescribe a constant specific mass density $\ov{\rho}_i := M_i/V_i$, $i \in \{1,2\}$, as well as a partial mass density $\rho_i := M_i/V$ which need not be spatially constant.  The mass density of the mixture $\rho$ is the sum of the two partial mass densities, and is related to other densities via the relations
\begin{align}\label{phi:c}
\rho_i = \ov{\rho}_i \phi_i = \rho c_i, \quad i \in \{1,2\}.
\end{align}
We now consider an order parameter $\phi$ defined as the difference in volume fractions $\phi := \phi_1 - \phi_2$, and assume the mixture density $\rho$ to be a function of $\phi$.  It is possible to consider the mixture density as a function of a different order parameter based on the difference in mass concentration $c := c_1 - c_2$, see for instance \cite{AGG,Lowengrub}. From the conditions $\phi_1 + \phi_2 = 1$ and $\phi_1, \phi_2 \geq 0$, we deduce that $\phi \in [-1,1]$ and obtain the relations 
\begin{align}\label{rhoi:phi}
\rho_1 = \frac{\ov{\rho}_1}{2}(1+\phi), \quad \rho_2 = \frac{\ov{\rho}_2}{2}(1-\phi),
\end{align}
leading to the algebraic equation for the mixture density $\rho = \hat{\rho}(\phi)$:
\begin{align}\label{hat:rho:def}
\hat{\rho}(\phi) = \frac{\ov{\rho}_1}{2} (1 + \phi) + \frac{\ov{\rho}_2}{2} (1-\phi) = \frac{\ov{\rho}_1 - \ov{\rho}_2}{2} \phi + \frac{\ov{\rho}_1 + \ov{\rho}_2}{2}
\end{align}
valid for $\phi \in [-1,1]$. Note that $\rho$ and $\hat{\rho}(\phi)$ coincide only when $\phi \in [-1,1]$.

Let $\bv_i$, $i \in \{1,2\}$ denote the velocity of the $i$th fluid. For the choice of the mixture velocity $\bv$, there are two candidates: The first is the volume averaged velocity \cite{AGG,Boyer,Ding} defined as
\begin{align}\label{va:velo}
\bv_v := \phi_1 \bv_1 + \phi_2 \bv_2,
\end{align}
which is divergence-free, leading to several analytical and numerical advantages. The second is the mass averaged velocity \cite{Aki,Lowengrub,SR:Zee} defined as
\begin{align}\label{ma:velo}
\bv_m := \frac{1}{\rho} (\rho_1 \bv_1 + \rho_2 \bv_2),
\end{align}
which is not solenoidal. In this work our model is based on the volume averaged velocity $\bv_v$, and thus we use the notation $\bu = \bv_v$ when there is no ambiguity. While it has been pointed out that in \cite{Eikelder,Yang} that the difference between these two can be quite significant, the advantage of the volume averaged model is that it is better amenable to further mathematical analysis compared to the mass averaged model. A derivation and analysis for the mass averaged model will be the subject of a future work.

\paragraph{Notation and convention.}
For the subsequent sections we make use of the Einstein summation convention and neglect the basis vector elements. 
\begin{itemize}
\item For a vector $\bv = (v_i)$ and a second order tensor $\bm{A} = (A_{ij})$, the gradient $\nabla \bv$ and divergence $\div \bm{A}$ are defined as
\[
(\nabla \bv)_{ij} = \pd_i v_j = \frac{\pd v_j}{\pd x_i}, \quad (\div \bm{A})_j = \pd_i A_{ij}.
\]
Meanwhile the vector $\bm{v} \cdot \bm{A}$ is defined as
\[
(\bm{v} \cdot \bm{A})_j = v_i A_{ij}.
\]
The Frobenius product $\bm{A}: \bm{B}$ of two second order tensors $\bm{A}$ and $\bm{B}$ is defined as
$\bm{A} : \bm{B} = A_{ij} B_{ij}$.
\item In three spatial dimensions, the entries of the third order Levi-Civita tensor $\bm{\eps} = (\eps_{ijk})$ are defined as
\[
\eps_{ijk} = \begin{cases}
1 & \text{ if } (i,j,k) \text{ is } (1,2,3), (2,3,1) \text{ or } (3,1,2), \\
-1 & \text{ if } (i,j,k) \text{ is } (3,2,1), (1,3,2) \text{ or } (2,1,3), \\
0 & \text{ if } i = j, \text{ or } j = k, \text{ or } k = i.
\end{cases}
\]
Then, in three spatial dimensions, the following properties are valid:
\begin{align}\label{eps:prop}
\eps_{ljk} \eps_{mjk} = 2 \delta_{lm}, \quad \eps_{jik} \eps_{jlm} = \delta_{il} \delta_{km} - \delta_{im} \delta_{kl}.
\end{align}
Furthermore, the cross product $\bm{a} \times \bm{b}$ between two vectors $\bm{a}$ and $\bm{b}$, as well as the curl of a vector $\bv$ are defined as
\[
(\bm{a} \times \bm{b})_j = \eps_{jkl} a_k b_l, \quad (\curl \bv)_j = \eps_{jkl} \pd_k v_l = (\nabla \times \bv)_j.
\]
Then, we have the following integration by parts formula involving the curl operator:
\begin{equation}\label{IBP:curl}
\begin{aligned}
\int_\Omega \curl \bm{a} \cdot \bm{b} \, dx & = \int_{\Omega} \bm{a} \cdot \curl \bm{b} \, dx + \int_{\pd \Omega} (\bm{a} \times \bm{b}) \cdot \bnu \, dS \\
& = \int_{\Omega} \bm{a} \cdot \curl \bm{b} \, dx - \int_{\pd \Omega}  (\bm{a} \times \bnu) \cdot \bm{b} \, dS.
\end{aligned}
\end{equation}
For a skew second order tensor $\bm{T}$, i.e., $T_{ij} = - T_{ji}$, its antisymmetric vector $\bm{T}_x$ is defined as
\begin{align}\label{defn:Tx}
(\bm{T}_x)_j = \eps_{jkl} T_{kl}.
\end{align}
\end{itemize}

\subsection{Balance laws}
Assuming no mass product or conversion between the constituent fluids, the mass balance for the individual mass density in local form reads as
\begin{align}\label{rho:i:eq}
\pd_t \rho_i + \div (\rho_i \bv_i) = 0 \quad \text{ for } i \in \{1,2\}.
\end{align}
We introduce diffusive velocities $\bm{J}_1$ and $\bm{J}_2$ with respect to $\bu$ by
\begin{equation}\label{vol:av:diff:velo}
\begin{aligned}
\bm{J}_1 & := \rho_1 (\bv_1 - \bu) = \rho_1 ((1 - \phi_1) \bv_1 - \phi_2 \bv_2) = \rho_1 \phi_2 (\bv_1 - \bv_2), \\
\bm{J}_2 &: = \rho_2(\bv_2 - \bu) = \rho_2 \phi_1 (\bv_2 - \bv_1),
\end{aligned}
\end{equation}
so that  \eqref{rho:i:eq} can be equivalently expressed as
\begin{align}\label{mass:bal:bu}
\pd_t \rho_1 + \div (\rho_1 \bu + \bm{J}_1) = 0, \quad \pd_t \rho_2 + \div (\rho_2 \bu + \bm{J}_2) = 0.
\end{align}
Dividing the individual mass balances \eqref{mass:bal:bu} by $\ov{\rho}_i$ respectively and taking the difference leads to the equation for $\phi$:
\begin{align}\label{bu:phi:equ}
\pd_t \phi + \div (\phi \bu) + \div \bm{J}_{\phi} = 0, \quad \bm{J}_\phi := \frac{\bm{J}_1}{\ov{\rho}_1} - \frac{\bm{J}_2}{\ov{\rho}_2}.
\end{align}
From the definition of $\bu = \phi_1 \bv_1 + \phi_2 \bv_2$ and the individual mass balance \eqref{rho:i:eq}, we infer the divergence-free condition for $\bu$:
\[
\div \bu = \div \Big ( \frac{\rho_1}{\ov{\rho}_1} \bv_1 + \frac{\rho_2}{\ov{\rho}_2} \bv_2 \Big ) = - \pd_t \Big ( \frac{\rho_1}{\ov{\rho}_1} + \frac{\rho_2}{\ov{\rho}_2} \Big ) = -\pd_t (\phi_1 + \phi_2) = 0.
\]
On the other hand, taking the sum of \eqref{mass:bal:bu} yields the equation for $\rho$:
\begin{align}\label{rho:bu:equ:original}
\pd_t \rho + \div (\rho \bu) + \div \bm{J} = 0, \quad \bm{J} := \bm{J}_1 + \bm{J}_2.
\end{align}
The connection between $\bm{J}$ and $\bm{J}_i$ can be obtained by noticing that 
\[
\frac{\bm{J}_1}{\ov{\rho}_1} + \frac{\bm{J}_2}{\ov{\rho}_2} = \phi_1 (\bv_1 - \bu) + \phi_2 (\bv_2 - \bu) = \bm{0},
\]
which leads to
\[
\bm{J} = \bm{J}_1 + \bm{J}_2 = \frac{\ov{\rho}_1 - \ov{\rho}_2}{2} \Big ( \frac{\bm{J}_1}{\ov{\rho}_1} - \frac{\bm{J}_2}{\ov{\rho}_2} \Big ) + \frac{\ov{\rho}_1 + \ov{\rho}_2}{2} \Big ( \frac{\bm{J}_1}{\ov{\rho}_1} + \frac{\bm{J}_2}{\ov{\rho}_2} \Big ) = \frac{\ov{\rho}_1 - \ov{\rho}_2}{2}  \bm{J}_{\phi}, 
\]
and allows us to express \eqref{rho:bu:equ:original} equivalently as
\begin{align}\label{rho:equ}
\pd_t \rho + \div (\rho \bu) + \frac{\ov{\rho}_1 - \ov{\rho}_2}{2} \div \bm{J}_\phi = 0.
\end{align}
In light of the mathematical analysis we perform in Section \ref{sec:analysis} we express \eqref{rho:equ} in another equivalent form:
\begin{equation}\label{rho:equ:alt}
\begin{aligned}
\pd_t \rho + \div (\rho \bu) & = - \rho'(\phi) \div \bm{J}_\phi = -\div (\rho'(\phi) \bm{J}_\phi) + \bm{J}_\phi \cdot \nabla \rho'(\phi) \\
& =: -\div (\rho'(\phi) \bm{J}_\phi) + R.
\end{aligned}
\end{equation}
Note that if $\phi \in [-1,1]$, then $\nabla \rho'(\phi) = \bm{0}$ and thus the additional term $R$ vanishes. We continue the model derivation without the assumption that $\phi \in [-1,1]$.

We now consider a pseudo-momentum $\bm{m} := \rho \bu$ based on the volume averaged velocity, which differs from the mixture linear momentum $\rho \bm{v}_m$ based on the mass averaged velocity consistent with mixture theory \cite{Bowen,Eikelder,SR:Zee}. Nevertheless, we postulate the pseudo-momentum satisfies the balance law
\begin{align}\label{pseudo:mom:bu}
\pd_t (\rho \bu) + \div (\rho \bu \otimes \bu) = \div (\bm{T} - p \I) + \bm{k},
\end{align}
where $\bm{k}$ is an additional contribution that vanishes when $\phi \in [-1,1]$. Its presence is to account for the effects when $\phi \notin [-1,1]$, which will be relevant in the mathematical analysis we perform in Section \ref{sec:analysis}. Then, with the help of \eqref{rho:equ:alt} yields the equivalent form
\begin{equation}\label{pseudo:mom:bu:2}
\begin{aligned}
\rho \nd \bu & = \rho (\pd_t \bu + (\bu \cdot \nabla) \bu) \\
& = \div \Big ( \bm{T} - p \I + \rho'(\phi) \bm{J}_\phi \otimes \bu \Big ) -  (\rho'(\phi) \bm{J}_\phi \cdot \nabla) \bu + \bm{k} - R \bu,
\end{aligned}
\end{equation}
where $\nd f := \pd_t f + \nabla f \cdot \bu$ denotes the material derivative with respect to the volume averaged velocity $\bu$. The additional term $(\bm{J} \cdot \nabla) \bu$ is a key feature of the model arising from the use of the pseudo-momentum. Then, for a test volume $V$ advected by the volume averaged velocity $\bu$, we postulate the balance of total angular momentum to be
\begin{equation}\label{bu:ang:mom:post}
\begin{aligned}
& \frac{d}{dt} \int_{V(t)} \rho \big ( \bm{x} \times \bu + \bm{j} \big) \, dx \\
& \quad = \int_{\pd V(t)} \bnu \cdot \bm{C} + \bm{x} \times \big [\bnu \cdot \big (\bm{T} - p \I \big )  \big ] \, dS + \int_{V(t)}  \bm{x} \times \bm{k} + \bm{g} \times \bu  + \bm{h} \, dx,
\end{aligned}
\end{equation}
where $\bm{x}$ is an arbitrary point in $\Omega$, $\bm{j}$ denotes the intrinsic angular momentum per unit mass, $\bm{C}$ is the couple stress tensor, $\bnu$ is the outer unit normal to $V$, $\bm{g}$ accounts for sources of angular momentum due to diffusion of the constituent fluids, while $\bm{h}$ is an additional contribution that vanishes when $\phi \in [-1,1]$. We further assume that $\bm{j} = j^* \bo$ where the vector field $\bo$ denotes the angular velocity of the rotation of fluid particles known as the micro-rotation, and $j^*$ is a constant scalar micro-inertia coefficient. Then, by the divergence theorem it follows that 
\begin{equation}\label{ang:mom:bal}
\begin{aligned}
& \frac{d}{dt} \int_{V(t)} \rho \big ( \bm{x} \times \bv + j^* \bo \big ) \, dx \\
& \quad = \int_{V(t)} \div \bm{C} + \bm{x} \times \big ( \div \big ( \bm{T} - p \I \big ) + \bm{k} \big ) + ( \bm{T} - p \I)_x + \bm{g} \times \bu + \bm{h} \, dx,
\end{aligned}
\end{equation}
where in the above calculation we have used the identity
\begin{align*}
\int_{\pd V} \bm{x} \times \big [ \bnu \cdot \bm{T} \big ] \, dS & = \int_{\pd V} \eps_{ijk} x_j T_{pk} \nu_p \, dS  = \int_V \eps_{ijk} \pd_p (x_j T_{pk}) \, dx  \\
& = \int_V \eps_{ijk} T_{jk} + \eps_{ijk} x_j \pd_p T_{pk} \, dx = \int_V \bm{T}_x + \bm{x} \times \div \bm{T} \, dx.
\end{align*}
Notice that the antisymmetric vector of a symmetric tensor vanishes, i.e., $(\I)_x = \bm{0}$. This yields the following the pointwise angular momentum balance law
\[
\bm{x} \times \nd{( \rho \bu)} + j^* \nd \rho \bo + j^* \rho \nd \bo = \div \bm{C} +  \bm{x} \times \big (\div  \big (\bm{T} - p \I \big) +\bm{k} \big ) + \bm{T}_{x} + \bm{g} \times \bu,
\]
and with the help of \eqref{rho:equ:alt} and \eqref{pseudo:mom:bu} this simplifies to 
\begin{align}\label{ang:mom:bu}
 \rho j^* \nd{\bo} = \div \Big ( \bm{C} + \bm{J} \otimes j^* \bo \Big ) + \bm{T}_x -  (\bm{J} \cdot \nabla) j^* \bo + \bm{g} \times \bu + \bm{h} - j^* R \bo.
\end{align}

\subsection{Energy dissipation inequality}
We consider the second law of thermodynamics in the form of a dissipation inequality \cite{Gurtinbk,Gurtin} for a bulk energy density $e$ of the form
\[
e = \frac{\rho}{2} |\bu|^2 + \frac{\rho j^*}{2} |\bo|^2 + \rho \psi,
\]
comprising of a Helmholtz free energy $\rho \psi$, the pseudo kinetic energy $\tfrac{\rho}{2} |\bu|^2$ and the micro-rotation energy $\tfrac{\rho}{2} j^* |\bo|^2$.  To account for the possibility of moving contact lines we introduce the total boundary energy comprising of a free energy density $\chi$ per unit area.  For an arbitrary test volume $V$ advected by the mixture velocity $\bu$ that has a nonempty intersection $\Sigma \cap \pd \Omega$ with the external boundary $\pd \Omega$, we assume an energy dissipation inequality of the form
\begin{align}\label{energy:diss:2}
\frac{d}{dt} \int_{V(t)} e \, dx + \frac{d}{dt} \int_{\Sigma(t)} \chi \, dS \leq W(V(t)) - M(V(t)) + H(\Sigma(t)),
\end{align}
where $W(V)$ denotes the total work done by the macro- and micro-stresses on $V$, $M(V)$ denotes the energy transported into $V$ by diffusion, and $H(\Sigma)$ denotes the total work done on $\Sigma$ and energy transport into $\Sigma$. Following \cite{AGG,GurtinD} we consider
\begin{align*}
W(V) & = \int_{\pd V \setminus \Sigma} [\bnu \cdot (\bm{T} - p \I  + \bm{J} \otimes \bu) ] \cdot \bu + [\bnu \cdot  (\bm{C} + \bm{J} \otimes j^* \bo)] \cdot \bo + \nd \phi \bm{\xi} \cdot \bm{\nu} \, dS, \\
M(V) & = \int_{\pd V \setminus \Sigma} \mu \bm{J}_\phi \cdot \bm{\nu} + \frac{1}{2} (|\bu|^2 + j^* |\bo|^2) \bm{J} \cdot \bnu \, dS, \quad H(\Sigma) = \int_{\pd \Sigma} \bm{t} \cdot \bmu \, dl,
\end{align*}
with a chemical potential $\mu$, a vectorial micro-stress $\bm{\xi}$, a general tangential flux $\bm{t}$ and unit conormal $\bmu$ of $\pd \Sigma$. Let $\nS$ and $\divS$ denote the surface gradient and surface divergence on $\Gamma := \pd \Omega$. Then, by the Reynold's transport theorem, the surface divergence theorem and surface Reynold's transport theorem (see e.g.~\cite{Dziuk}) we infer from \eqref{energy:diss:2}
\begin{equation}\label{bu:diss:ineq}
\begin{aligned}
& \int_{V(t)} \nd{e} - \div \Big ( ( \bm{T} - p \I) \bu + \bm{C} \bo + \nd{\phi} \bm{\xi} - \mu \bm{J}_\phi + \frac{1}{2}(|\bu|^2 + j^* |\bo|^2) \bm{J} \Big ) \, dx  \\
& \quad + \int_{\Sigma(t)} \nd{\chi} + \chi \divS \bu  + [\bnu \cdot (\bm{T} - p \I + \bm{J} \otimes \bu) ] \cdot \bu + [\bnu \cdot (\bm{C} + \bm{J} \otimes j^* \bo)]\cdot \bo \, dS \\
& \quad + \int_{\Sigma(t)} \nd \phi \bm{\xi} \cdot \bnu - \mu \bm{J}_\phi \cdot \bnu - \divS \bm{t} - \frac{1}{2} ( |\bu|^2 + j^* |\bo|^2 ) \bm{J} \cdot \bnu \, dS  \leq 0.
\end{aligned}
\end{equation}

\subsection{Specific choice of free energy}\label{sec:CA}
We make the following constitutive assumption for the free energy density $\rho \psi$
\begin{align}\label{GL}
\rho\psi = \widehat{\rho \psi}(\phi, \nabla \phi).
\end{align}
The standard choice is a Ginzburg--Landau type functional
\begin{align}\label{GL:eg}
\widehat{\rho \psi}(\phi, \nabla \phi) = \sigma \eps A(\nabla \phi) + \frac{\sigma}{\eps} F(\phi),
\end{align}
where $\eps > 0$ is a small parameter associated to the thickness of interfacial layers, $\sigma$ is related to the surface energy density, $A : \R^d \to \R$ is a positively two-homogeneous function, i.e., $A(\lambda \bm{p}) = \lambda^2 A(\bm{p})$ for all $\lambda > 0$ and $\bm{p} \in \R^d$, and $F$ is a double well potential. The commonly used example is the isotropic energy density
\begin{align}\label{isotropic}
A(\bm{p}) = \frac{1}{2} |\bm{p}|^2,
\end{align}
so that $A(\nabla \phi) = \frac{1}{2} |\nabla \phi|^2$, paired with either the quartic potential $F(s) = \frac{1}{4}(s^2-1)^2$, or the thermodynamically relevant logarithmic potential
\begin{align}\label{log:pot}
F(\phi) = \frac{\theta}{2} ((1-\phi) \log(1-\phi) + (1+\phi) \log(1+\phi)) - \frac{\theta_c}{2} \phi^2,
\end{align}
with constants $0 < \theta < \theta_c$. Also relevant is the non-smooth double obstacle potential obtained from the logarithmic potential in the \emph{deep quench} limit $\theta \to 0$:
\[
F(\phi) = \begin{cases}
\frac{1}{2}(1-\phi^2) & \text{ if } |\phi| \leq 1, \\
+\infty & \text{ otherwise}.
\end{cases}
\]
In the case of the isotropic gradient energy \eqref{isotropic}, an alternative choice to \eqref{GL} is the following proposed in \cite{Gia1,Gia2}:
\begin{align}\label{nlGL}
\rho \psi = \widehat{\rho \psi}(\phi) = \frac{\sigma \eps}{4} \int_\Omega K(x-y) |\phi(x) - \phi(y)|^2 \, dy + \frac{\sigma}{\eps} F(\phi),
\end{align}
that was derived from a hydrodynamic limit of a microscopic model for a lattice gas evolving via a Poisson nearest-neighbor process.  In the above, $K : \Omega \times \Omega \to \R$ is a suitable symmetric convolution kernel, such that when rescaled as 
\[
K_\kappa (x-y) := \frac{\rho_\kappa(|x-y|)}{|x-y|^2}
\] 
where $(\rho_\kappa)_{\kappa > 0}$ is a suitable family of mollifiers converging to a Dirac delta, then it has been shown in \cite{Ponce} via the framework of Gamma-convergence that
\[
\lim_{\kappa \searrow 0} \frac{1}{4} \int_{\Omega \times \Omega} K_\kappa(x-y)|\phi_\kappa(x) - \phi_\kappa(y)|^2 \, dx \, dy = \begin{cases}
\int_\Omega \frac{1}{2} |\nabla \phi|^2 \,dx & \text{ if } \phi \in H^1(\Omega), \\[1ex]
+\infty & \text{ otherwise},
\end{cases}
\]
whenever $\phi_\kappa \to \phi$ in $L^2(\Omega)$. Hence, one may also view \eqref{nlGL} as an approximation of \eqref{GL:eg} with \eqref{isotropic}.

For the boundary free energy density $\chi$ we make the following constitutive assumption:
\begin{align}\label{chi:defn}
\chi = \hat{\chi}(\phi, \nS \phi) = \hat{G}(\phi) + \beta B(\nS \phi)
\end{align}
with constant $\beta \geq 0$ and a positively two-homogeneous function $B: \R^d \to \R$. For given constant surface tensions $\gamma_i$, $i =1,2$, between the $i$th fluid and the external boundary, we can consider $\hat{G}$ as an interpolation function between the two values. Examples include (see \cite{QianPhys,Qian})
\begin{align}\label{hatG:eg}
\hat{G}(s) = \frac{\gamma_1 - \gamma_2}{2}s + \frac{\gamma_1+\gamma_2}{2} \; \text{ or } \;  \hat{G}(s) = \frac{\gamma_1 - \gamma_2}{2} \sin \big ( \frac{\pi s}{2} \big )+ \frac{\gamma_1 + \gamma_2}{2},
\end{align}
for $s \in [-1,1]$ and similar to \eqref{isotropic} a commonly used example for $B$ is the isotropic energy density
\[
B(\bm{p}) = \frac{1}{2} |\bm{p}|^2.
\]
For the forthcoming calculations we consider a bulk gradient energy $A$ satisfying the symmetric relation
\begin{align}\label{A':sym}
A'(\bm{p}) \otimes \bm{p} = \bm{p} \otimes A'(\bm{p}) \quad \text{ for all } \bm{p} \in \R^d,
\end{align}
which in turn implies that
\[
|\bm{p}|^2 |A'(\bm{p})|^2 = (A'(\bm{p}) \cdot \bm{p})^2
\]
after taking a scalar product with the vector $\bm{p}$. Hence, $A'(\bm{p}) = a(\bm{p}) \bm{p}$ for some function $a$ that is positively zero-homogeneous, i.e., $a(\lambda \bm{p}) = a(\bm{p})$ for all $\bm{p} \in \R^d$ and $\lambda > 0$.

\subsection{Constitutive assumptions}
A short calculation employing \eqref{rho:equ:alt}, \eqref{pseudo:mom:bu:2} and \eqref{ang:mom:bu} shows (recalling the relation $\bm{J} = \rho'(\phi) \bm{J}_\phi$)
\begin{align*}
& \pd_t \Big ( \frac{\rho}{2}|\bu|^2 + \frac{\rho j^*}{2} |\bo|^2 \Big ) + \bu \cdot \nabla \Big ( \frac{\rho}{2}|\bu|^2 + \frac{\rho j^*}{2} |\bo|^2 \Big ) \\
& \quad = \nd{\rho} \frac{|\bu|^2}{2} + \rho \nd{\bu} \cdot \bu +  \nd{\rho} \frac{ j^* |\bo|^2}{2} + \rho j^*\nd{\bo} \cdot \bo \\
& \quad =  \div \Big ( (\bm{T} - p \I) \bu + \bm{C} \bo + \frac{1}{2}(|\bu|^2 + j^* |\bo|^2) \bm{J} \Big ) - \nabla \bu : \Big ( \bm{T} +  \bm{J} \otimes \bu \Big ) + \bu \cdot \big (\bm{k} - \frac{1}{2} R \bu \big ) \\
& \qquad - \nabla \bo : \Big ( \bm{C} + \bm{J} \otimes j^*\bo \Big ) + \big (\bm{T}_x  + \bm{g} \times \bu + \bm{h} - \frac{1}{2} j^* R \bo \big ) \cdot \bo,
\end{align*}
where we have used $\div \bu = \nabla \bu : \I = 0$. Then, denoting by $B' : \R^d \to \R^d$ as the gradient of $B(\cdot)$, we compute the following identities 
\begin{align}
\nd{[\widehat{\rho\psi}]} &  = \frac{\sigma}{\eps} F'(\phi) \nd{\phi} + \sigma \eps A'(\nabla \phi) \cdot \nabla \nd \phi - \sigma \eps (A'(\nabla \phi) \otimes \nabla \phi) : \nabla \bu, \label{md:free}
\end{align}
and
\begin{align*} \nd{[\hat{\chi}]} & := \pd_t \hat{\chi} + \bu \cdot \nS \hat{\chi} \\
& = \hat{G}'(\phi) \nd{\phi} + \beta B'(\nS \phi) \cdot \nS \nd{\phi} - \beta (B'(\nS \phi) \otimes \nS \phi) : \nS \bu, \\
 & = \Big ( \hat{G}'(\phi) - \beta \divS B'(\nS \phi) \Big ) \nd \phi + \beta \divS (B'(\nS \phi) \nd \phi) \\
 & \quad - \beta(B'(\nS \phi) \otimes \nS \phi) : \nS \bu, 
\end{align*}
and 
\begin{align*} \hat{\chi} \divS \bu & =  \divS (\hat{\chi} \bu) - \nS \hat{\chi} \cdot \bu \\
 & = \divS( \hat{\chi} \bu - \beta (\bu \cdot \nS \phi) B'(\nS \phi)) - \bu \cdot \Big (\hat{G}'(\phi) - \beta \divS (B'(\nS \phi)) \Big ) \nS \phi \\
 & \quad + \beta  (B'(\nS \phi) \otimes \nS \phi) :\nS \bu.
\end{align*}
Hence, together with \eqref{bu:phi:equ} we can simplify \eqref{bu:diss:ineq} to
\begin{equation}\label{bu:diss:ineq:2}
\begin{aligned}
& \int_{V(t)}  - \nabla \bu : \Big ( \bm{T} + \bm{J} \otimes \bu + \sigma \eps A'(\nabla \phi) \otimes \nabla \phi \Big ) - \nabla \bo: \Big ( \bm{C} + \bm{J} \otimes j^* \bo \Big ) \, dx \\
& \quad  + \int_{V(t)} \Big ( \frac{\sigma}{\eps} F'(\phi) - \div \bm{\xi} - \mu \Big ) \nd{\phi} + \nabla \mu \cdot \bm{J}_\phi + \Big ( \sigma \eps A'(\nabla \phi) - \bm{\xi} \Big ) \cdot \nabla \nd{\phi} \, dx \\
& \quad + \int_{V(t)} \big ( \bm{k} - \frac{1}{2} R \bu \big ) \cdot \bu + \big (\bm{T}_x + \bm{g} \times \bu + \bm{h} - \frac{1}{2} j^* R \bo \big )\cdot \bo \, dx \\
& \quad + \int_{\Sigma(t)} \Big ( \hat{G}'(\phi) - \beta \divS (B'(\nS \phi)) + \bm{\xi} \cdot \bm{\nu} \Big ) \nd \phi \\
& \quad + \int_{\Sigma(t)} \divS \Big ( \hat{\chi} \bu + \beta \pd_t \phi B'(\nS \phi) - \bm{t} \Big )  \, dS  \\
& \quad + \int_{\Sigma(t)} \Big ( \bm{\nu} \cdot (\bm{T} - p \I + \bm{J} \otimes \bu) - \big ( \hat{G}'(\phi) - \beta \divS (B'(\nS \phi)) \big ) \nS \phi \Big ) \cdot \bu \, dS \\
& \quad + \int_{\Sigma(t)} [\bm{\nu} \cdot (\bm{C} + \bm{J} \otimes j^* \bo)] \cdot \bo - \mu \bm{J}_\phi \cdot \bm{\nu} - \frac{1}{2} ( |\bu|^2 + j^* |\bo|^2 ) \bm{J} \cdot \bnu \, dS \leq 0.
\end{aligned}
\end{equation}
From the arbitrariness of the test volume $V(t)$ we infer the local inequalities
\begin{equation}\label{bu:bulk:diss}
\begin{aligned}
- \mathcal{D} & := - \nabla \bu : \Big ( \bm{T} + \bm{J} \otimes \bu + \sigma \eps A'(\nabla \phi) \otimes \nabla \phi \Big ) - \nabla \bo : \Big ( \bm{C} + \bm{J} \otimes  j^*\bo \Big ) \\
& \quad + \Big ( \frac{\sigma}{\eps} F'(\phi) - \div \bm{\xi} - \mu \Big ) \nd{\phi} + \Big ( \sigma \eps A'(\nabla \phi) - \bm{\xi} \Big ) \cdot \nabla \nd{\phi} \\
& \quad  + \nabla \mu \cdot \bm{J}_\phi + \big ( \bm{k} - \frac{1}{2} R \bu \big ) \cdot \bu + \big ( \bm{T}_x + \bm{g} \times \bu + \bm{h} - \frac{1}{2} j^* R \bo\big ) \cdot \bo \leq 0,
\end{aligned}
\end{equation}
and 
\begin{equation}\label{bu:surf:diss}
\begin{aligned}
- \mathcal{D}_\Gamma & := \Big ( \hat{G}'(\phi) - \beta \divS (B'(\nS \phi)) + \bm{\xi}\cdot \bm{\nu} \Big ) \nd \phi \\
& \quad + \divS \Big ( \hat{\chi} \bu + \beta \pd_t \phi B'(\nS \phi) - \bm{t} \Big ) \\
& \quad +  \Big ( \bm{\nu} \cdot (\bm{T} - p\I + \bm{J} \otimes \bu) - \big ( \hat{G}'(\phi) - \beta \divS (B'(\nS \phi)) \big ) \nS \phi \Big ) \cdot \bu \\
& \quad + [\bm{\nu} \cdot (\bm{C} + \bm{J} \otimes j^* \bo)] \cdot \bo - \mu \bm{J}_\phi \cdot \bm{\nu}  - \frac{1}{2}(|\bu|^2 + j^* |\bo|^2) \bm{J} \cdot \bm{\nu} \leq 0.
\end{aligned}
\end{equation}
We define new tensors $\bm{S} = \bm{T} + \bm{J} \otimes \bu$ and $\bm{E} = \bm{C} + \bm{J} \otimes j^* \bo$, and introduce the decomposition $\bm{S} = \bm{S}_s + \bm{S}_a$ and $\bm{E} = \bm{E}_s + \bm{E}_a$ into a sum of their symmetric and skew parts, where for a tensor $\bm{A}$ we define
\[
\bm{A}_s = \frac{1}{2}(\bm{A} + \bm{A}^{\top}), \quad \bm{A}_a = \frac{1}{2}(\bm{A} - \bm{A}^{\top}).
\]
Furthermore, we introduce the notation
\[
\D \bv = (\nabla \bu)_s, \quad \W \bv = (\nabla \bu)_a, \quad 
\D \bo = (\nabla \bo)_s, \quad \W \bo = (\nabla \bo)_a,
\]
so that $\nabla \bv = \D \bv + \W \bv$ and $\nabla \bo = \D \bo + \W \bo$, and define a second order skew tensor $\bm{\Omega}$ associated to $\bo$ as
\begin{align}\label{skew:ten}
\bm{\Omega}_{jk} = \eps_{ijk} \omega_i.
\end{align}
In particular for any vector $\bm{v} \in \R^d$,
\[
\bm{\Omega} \bm{v} = - \bo \times \bm{v}.
\]
Then, the term $\bm{T}_x \cdot \bo$ can be expressed as
\begin{equation}\label{skew:rel}
\begin{aligned}
\bm{T}_x \cdot \bo & = (\bm{T}_a)_x \cdot \bo = \eps_{ijk} (\bm{T}_a)_{jk} \omega_i = \bm{\Omega} : \bm{T}_a  \\
&  = \bm{\Omega} : (\bm{S} - \bm{J} \otimes \bu) = \bm{S}_a : \bm{\Omega} - \bo \cdot (\bm{J} \times \bu),
\end{aligned}
\end{equation}
where we have used $\eps_{ijk} = - \eps_{ikj}$ and recalled that the antisymmetric vector of a symmetric second order tensor vanishes. From this \eqref{bu:bulk:diss} simplifies to
\begin{equation}\label{bu:bulk:diss:2}
\begin{aligned}
-\mathcal{D} & = - \D \bu : \Big ( \bm{S}_s + \sigma \eps A'(\nabla \phi) \otimes \nabla \phi \Big ) - \bm{S}_a : \Big (\W \bu - \bm{\Omega}\Big ) + \bo \cdot ((\bm{g} - \bm{J}) \times \bu) \\
& \quad - \D \bo : \bm{E}_s - \W \bo : \bm{E}_a + \Big ( \frac{\sigma}{\eps} F'(\phi) - \div \bm{\xi} - \mu \Big ) \nd{\phi} + \big (\bm{k} - \frac{1}{2} R \bu \big ) \cdot \bu  \\
& \quad + \big ( \bm{h} - \frac{1}{2}j^* R \bo \big ) \cdot \bo + \Big ( \sigma \eps A'(\nabla \phi) - \bm{\xi} \Big ) \cdot \nabla \nd{\phi} + \nabla \mu \cdot \bm{J}_\phi  \leq 0,
\end{aligned}
\end{equation}
where we have invoked the symmetric relation \eqref{A':sym} for the tensor $\sigma \eps A'(\nabla \phi) \otimes \nabla \phi$. Choosing constitutive relations
\begin{equation*}
\begin{alignedat}{3}
\bm{S}_s & = \widehat{\bm{S}}_s(\phi, \nabla \phi, \nabla \bu), \quad &&\bm{S}_a && = \widehat{\bm{S}}_a( \phi, \nabla \phi, \nabla \bu, \bm{\Omega}), \\
\bm{E}_s & = \widehat{\bm{E}}_s(\phi, \nabla \phi, \nabla \bo), \quad  && \bm{E}_a && = \widehat{\bm{E}}_a(\phi, \nabla \phi, \nabla \bo), \\
\bm{\xi} & = \widehat{\bm{\xi}}(\phi, \nabla \phi), \quad &&\bm{J}_\phi && = \widehat{\bm{J}_\phi}(\phi, \nabla \phi, \mu, \nabla \mu), \\
\bm{g} & = \widehat{\bm{g}}(\phi, \nabla \phi, \mu, \nabla \mu), \quad &&  \bm{k} && = \widehat{\bm{k}}(\phi, \nabla \phi, \mu, \nabla \mu , \bu), \quad \bm{h} = \widehat{\bm{h}}(\phi, \nabla \phi, \mu, \nabla \mu , \bo),
\end{alignedat}
\end{equation*}
then by a Coleman--Noll procedure, to avoid the violation of \eqref{bu:bulk:diss:2} for arbitrary $\bu$, $\bo$, $\nd \phi$ and $\nabla \nd \phi$, we consider the constitutive assumptions
\begin{equation}\label{bu:constit}
\begin{aligned}
\bm{\xi} & = \sigma \eps A'(\nabla \phi) = \sigma \eps a(\nabla \phi) \nabla \phi, \quad \mu = \frac{\sigma}{\eps} F'(\phi) - \sigma \eps \div (a(\nabla \phi) \nabla \phi), \\
\bm{J}_\phi & = - m(\phi) \nabla \mu, \quad \bm{g} = \bm{J} = -\rho'(\phi) m(\phi) \nabla \mu \\[1ex]
\bm{S}_s  & =  2\eta(\phi) \D \bu -  \sigma \eps a(\nabla \phi) \nabla \phi \otimes \nabla \phi , \quad \bm{S}_a = 2 \eta_r(\phi) (\W \bu - \bm{\Omega} ),  \\[1ex]
\bm{E}_s  & = c_0(\phi) (\div \bo) \I + 2c_d(\phi) \D \bo , \quad \bm{E}_a = 2 c_a(\phi) \W \bo, \\
\bm{k} & = \frac{1}{2} R \bu = -\frac{1}{2} [m(\phi) \nabla \mu \cdot \nabla \rho'(\phi) ] \bu, \\
\bm{h} & = \frac{1}{2}j^* R \bo = -\frac{j^*}{2} [m(\phi) \nabla \mu \cdot \nabla \rho'(\phi) ] \bo,
\end{aligned}
\end{equation}
with scalar viscosity functions $\eta$, $\eta_r$, $c_0$, $c_d$, $c_a$, and scalar mobility function $m$.  Non-positivity of $-\mathcal{D}$ can be guaranteed provided $m \geq 0$, $\eta_r \geq 0$, along with the classical conditions \cite[Chap.~1, Sec.~4.2]{Luka}
\begin{align}\label{Eringen:cond}
c_d \geq 0, \quad c_a + c_d \geq 0, \quad 3 c_0 + 2 c_d \geq 0, \quad |c_d - c_a| \leq c_d + c_a.
\end{align}
\begin{remark}
The restrictions in \eqref{Eringen:cond} for the coefficients arise from enforcing
\[
c_0 \pd_k \omega_k \pd_l \omega_l + c_d (\pd_k \omega_l + \pd_l \omega_k) \pd_k \omega_l + c_a( \pd_k \omega_l - \pd_l \omega_k) \pd_k \omega_l \geq 0
\]
for arbitrary values of $\nabla \bo$. In three spatial dimensions, the above can be expressed as a quadratic form over a nine-dimensional space, see \cite[p.~10]{Eringen} with $\alpha_{\nu} = c_0$, $\beta_{\nu} = c_d - c_a$, $\gamma_{\nu} = c_d + c_a$ for more details.
\end{remark}
Then, recalling \eqref{defn:Tx} we obtain
\begin{align*}
\bm{T}_x + \bm{g} \times \bu & = (\bm{S}_a - \bm{J} \otimes \bu)_x  + \bm{J} \times \bu = (\bm{S}_a)_x  \\
& = \eps_{ljk} \eta_r(\phi) (\pd_j v_k - \pd_k v_j - 2\eps_{mjk} \omega_m) =  \eta_r(\phi)(2  \curl \bu - 4 \bo).
\end{align*}
Subsequently, \eqref{bu:surf:diss} now simplifies to 
\begin{equation}\label{bu:surf:diss:simpli}
\begin{aligned}
- \mathcal{D}_\Gamma & = \Big ( \hat{G}'(\phi) - \beta \divS(B'(\nS \phi)) + \sigma \eps A'(\nabla \phi) \cdot \bnu \Big ) \nd \phi \\
& \quad + \divS \Big ( \hat{\chi} \bu + \beta \pd_t \phi B'(\nS \phi) - \bm{t} \Big ) \\
& \quad +  \Big ( \bm{\nu} \cdot (\bm{S}_s + \bm{S}_a - p\I ) - \big ( \hat{G}'(\phi) - \beta \divS(B'(\nS \phi)) \big ) \nS \phi \Big ) \cdot \bu \\
& \quad + [\bm{\nu} \cdot (\bm{E}_s + \bm{E}_a)] \cdot \bo + m(\phi) \mu \pdnu \mu \\
& \quad + \frac{\rho'(\phi)}{2}(|\bu|^2 + j^* |\bo|^2) m(\phi) \pdnu \mu \leq 0,
\end{aligned}
\end{equation}
where $\pdnu f = \nabla f \cdot \bm{\nu}$ is the normal derivative of $f$ on the boundary. To avoid violation of \eqref{bu:surf:diss:simpli} for arbitrary $\nd \phi$, $\bu$ and $\bo$, we consider the constitutive assumption
\[
\bm{t} = \hat{\chi} \bu + \beta \pd_t \phi B'(\nS \phi)
\]
and the boundary conditions
\begin{align}\label{phi:bc}
\nd \phi = - \gamma(\phi) \Big ( \hat{G}'(\phi) - \beta \divS (B'(\nS \phi)) + \sigma \eps A'(\nabla \phi) \cdot \bnu \Big ), \quad m(\phi) \pdnu \mu = 0,
\end{align}
with prefactor $\gamma(\phi) > 0$. For the velocity $\bu$ and micro-rotation $\bo$, to include the dynamics of a moving contact line in our model, we first consider the no-penetration conditions:
\begin{align}\label{no:pene}
\bu \cdot \bnu = 0, \quad \bo \cdot \bnu = 0.
\end{align}
Denoting by $\bu_\tau := \bu - (\bu \cdot \bnu) \bnu$ as the tangential component of $\bu$ on $\pd \Omega$, and likewise for $\bo_\tau$, recalling
\begin{align*}
\bm{S}_s + \bm{S}_a & = 2\eta(\phi) \D \bu + 2 \eta_r(\phi)( \W \bu - \bm{\Omega}) - \sigma \eps A'(\nabla \phi) \otimes \nabla \phi, \\
\bm{E}_s + \bm{E}_a & = c_0(\phi) (\div \bo) \mathbb{I} + 2 c_d(\phi) \D \bo + 2 c_a(\phi) \W \bo,
\end{align*}
then \eqref{bu:surf:diss:simpli} can be expressed as
\begin{align*}
- \mathcal{D}_\Gamma & = - \frac{1}{\gamma(\phi)} |\nd \phi|^2 - \big ( \hat{G}'(\phi) - \beta \divS(B'(\nS \phi)) + \sigma \eps A'(\nabla \phi) \cdot \bnu \big ) \nS \phi \cdot \bu_\tau \\
& \quad + \Big ( 2\bm{\nu} \cdot [\eta(\phi) \D \bu + \eta_r(\phi) \W \bu - \eta_r(\phi) \bm{\Omega}] \Big )_\tau \cdot \bu_\tau\\
& \quad + \Big (2\bm{\nu} \cdot (c_a(\phi) \D \bo + c_d(\phi) \W \bo) \Big )_\tau \cdot \bo_\tau, 
\end{align*} 
where we have used that $(\bm{\nu} \cdot \mathbb{I})_\tau = \bm{0}$. Hence, non-positivity of $-\mathcal{D}_\Gamma$ can be guaranteed for arbitrary $\bu$ and $\bo$ if we consider the boundary conditions
\begin{align}
c_u(\phi) \bu_\tau & =  - \Big ( 2 \bnu \cdot  [\eta(\phi) \D \bu + \eta_r(\phi) \W \bu - \eta_r(\phi) \bm{\Omega} ] \Big) _\tau - \frac{1}{\gamma(\phi)} \nd \phi \nS \phi \label{u:bc} \\
c_w(\phi) \bo_\tau & = - \Big ( 2 \bnu \cdot [c_a(\phi)  \D \bo + c_d(\phi) \W\bo]  \Big )_\tau, \label{w:bc}
\end{align}
with nonnegative prefactors $c_u(\phi)$ and $c_w(\phi)$. 

Hence, a diffuse interface model for two-phase micropolar flow based on the volume averaged velocity $\bu$, pseudo-momentum $\rho \bu$ and pseudo-kinetic energy $\frac{\rho}{2}|\bu|^2$ with moving contact line dynamics reads as
\begin{subequations}\label{bu:PF:model}
\begin{alignat}{2}
\div \bu & = 0, \\[1ex]
\nd{\phi} & = \div (m(\phi) \nabla \mu), \\[1ex]
 \mu & = \frac{\sigma}{\eps} F'(\phi) - \sigma \eps \div (a(\nabla \phi) \nabla \phi), \\[1ex] 
\nd \rho & = \div (\rho'(\phi) m(\phi) \nabla \mu)  + R, \quad R = - m(\phi) \nabla \mu \cdot \nabla \rho'(\phi), \label{PF:rho} \\[1ex]
 \rho \nd{u} & = - \nabla p - \div \Big ( \sigma \eps a(\nabla \phi) \nabla \phi \otimes \nabla \phi \Big ) + \div (2 \eta(\phi) \D \bu + 2 \eta_r(\phi) \W \bu) \label{PF:lin:mom} \\[1ex] 
\notag & \quad + 2 \curl(\eta_r(\phi) \bo) + \rho'(\phi) \big ( m(\phi) \nabla \mu \cdot \nabla \big ) \bu - \frac{1}{2} R \bu , \\[1ex]
 \rho j^* \nd{\bo} & = \div (c_0(\phi) (\div \bo) \I + 2c_d(\phi) \D \bo +2 c_a(\phi) \W \bo)\label{PF:ang:mom} \\[1ex]
\notag  & \quad  + 2 \eta_r(\phi)(\curl \bv - 2 \bo) + j^* \rho'(\phi) \big ( m(\phi) \nabla \mu \cdot \nabla \big ) \bo - \frac{j^*}{2} R \bo, 
\end{alignat}
\end{subequations}
with boundary conditions
\begin{subequations}\label{bu:PF:GNBC}
\begin{alignat}{2}
& m(\phi) \pdnu \mu = 0, \quad \bu \cdot \bnu = 0, \quad \bo \cdot \bnu = 0  \\[1ex]
& \nd{\phi} = - \gamma(\phi) \big (\hat{G}'(\phi) - \beta \divS(B'(\nS \phi)) + \sigma \eps a(\nabla \phi) \pdnu \phi \big ), \\[1ex]
& c_u(\phi) \bu_\tau +  \Big ( 2 \bnu \cdot [\eta(\phi) \D \bu + \eta_r(\phi) \W \bu - \eta_r(\phi) \bm{\Omega}]  \Big )_\tau = -\frac{1}{\gamma(\phi)} \nd \phi \nS \phi, \\[1ex]
& c_w(\phi) \bo_\tau + \Big ( 2 \bnu \cdot [c_d(\phi) \D \bo + c_a(\phi) \W \bo]\Big )_\tau  = \bm{0}.
\end{alignat}
\end{subequations}

\begin{remark}
If $\phi \in [-1,1]$ holds, then $\rho'(\phi)$ is simply the constant $\frac{\overline{\rho}_1 - \overline{\rho}_2}{2}$, which leads to $R = 0$ in \eqref{PF:rho}, \eqref{PF:lin:mom} and \eqref{PF:ang:mom}, i.e.,
\begin{subequations}\label{PF:phi:alt}
\begin{alignat}{2}
\nd \rho & = \tfrac{\overline{\rho}_1 - \overline{\rho}_2}{2}  \div (m(\phi) \nabla \mu), \label{PF:rho:alt} \\[1ex]
 \rho \nd{u} & = - \nabla p - \div \Big ( \sigma \eps a(\nabla \phi) \nabla \phi \otimes \nabla \phi \Big ) + \div (2 \eta(\phi) \D \bu + 2 \eta_r(\phi) \W \bu) \label{PF:lin:mom:alt} \\[1ex] 
\notag & \quad + 2 \curl(\eta_r(\phi) \bo) + \tfrac{\overline{\rho}_1 - \overline{\rho}_2}{2}  \big ( m(\phi) \nabla \mu \cdot \nabla \big ) \bu , \\[1ex]
 \rho j^* \nd{\bo} & = \div (c_0 (\div \bo) \I + 2c_d \D \bo +2 c_a \W \bo) + 2 \eta_r(\phi)(\curl \bv - 2 \bo)  \label{PF:ang:mom:alt} \\[1ex]
\notag  & \quad + j^* \tfrac{\overline{\rho}_1 - \overline{\rho}_2}{2} \big ( m(\phi) \nabla \mu \cdot \nabla \big ) \bo.
\end{alignat}
\end{subequations}
In particular, the model equations \eqref{bu:PF:model} remains valid even if the order parameter $\phi$ lies outside the physically relevant interval $[-1,1]$. These forms for the linear and angular momentum balances will play a crucial role in the analysis performed in Section \ref{sec:analysis}.
\end{remark}

\begin{remark}
Assuming $\phi \in [-1,1]$ holds, so that $R = 0$, and in the absence of micro-rotation effects, i.e., $\bo = \bm{0}$, $\eta_r = 0$, as well as isotropic gradient energy \eqref{isotropic}, i.e., $a(\nabla \phi) = 1$, the system \eqref{bu:PF:model} with \eqref{PF:phi:alt} reduces to the model of Abels, Garcke and Gr\"un \cite{AGG}. If, in addition, $\beta = 0$, then due to the symmetry of $\D \bv$ the equations \eqref{phi:bc} and \eqref{u:bc} simplify to
\[
\nd \phi = - \gamma \big ( \hat{G}'(\phi) + \sigma \eps \pdnu \phi \big ), \quad \delta \bv_\tau = - \big ( 2 \eta(\phi) [\D \bv] \bnu \big )_\tau + \big ( \sigma \eps \pdnu \phi + \hat{G}'(\phi) \big ) \nS \phi.
\] 
Together with the no-penetration condition $\bu \cdot \bnu = 0$ and the no-flux condition $m(\phi) \pdnu \mu = 0$ we recover the generalized Navier boundary condition for moving contact lines, see e.g.~\cite[(4.4)-(4.5)]{Qian}.
\end{remark}

\begin{remark}
In our calculation we made the implicit assumption that the external boundary $\pd \Omega$ is stationary. We refer the reader to \cite{Aero} where the authors considered a more general class of boundary conditions for the micro-rotations where the external boundary may be non-stationary and moves with a (prescribed) velocity field $\bu
_{\pd \Omega}$. For instance, boundary conditions of the following forms are proposed:
\[
\bm{C} \bnu = \alpha (\bo - \tfrac{1}{2} \curl \bu_{\pd \Omega}), \quad \text{ or } \quad \alpha (\bo - \tfrac{1}{2} \curl \bu_{\pd\Omega}) = \curl \bu - 2 \bo
\]
with a nonnegative constant $\alpha$. The formal limit $\alpha \to \infty$ in both examples yields $\bo = \tfrac{1}{2} \curl \bu_{\pd \Omega}$, which states that the micro-rotation at the boundary matches the vorticity induced by the rotating boundary. On the other hand, in the formal limit $\alpha \to 0$, we obtain $\bm{C} \bnu = \bm{0}$ in the first example and $\bo = \frac{1}{2} \curl \bu$ in the second example, see e.g.~\cite{Kirwan} for a further discussion. The latter states that the skew part of the stress tensor vanishes on the boundary, see \cite[Chap.~1, Sec.~5]{Luka}.
\end{remark}

\begin{remark}[Reformulation of the linear pseudo-momentum balance]\label{rem:mom}
Using the identities
\begin{align*}
\sigma \eps \div (A'(\nabla \phi) \otimes \nabla \phi) = \nabla \Big ( \frac{\sigma}{\eps} F(\phi) + \sigma \eps A(\nabla \phi) \Big ) - \mu \nabla \phi,
\end{align*}
the balance of linear pseudo-momentum \eqref{PF:lin:mom} can be expressed in a simpler form:
\begin{equation}\label{mom:alt}
\begin{aligned}
 \rho \nd{u} & = - \nabla \tilde{p} + \mu \nabla \phi + \div (2 \eta(\phi) \D \bu + 2 \eta_r(\phi) \W \bu) \\
  & \quad + 2 \curl(\eta_r(\phi) \bo) + \rho'(\phi) \big ( m(\phi) \nabla \mu \cdot \nabla \big ) \bu - \frac{1}{2}R \bu,
\end{aligned}
\end{equation}
with a rescaled pressure $\tilde{p} = p + \frac{\sigma}{\eps} F(\phi) + \sigma \eps A(\nabla \phi)$. Alternatively, by adding \eqref{mom:alt} with \eqref{PF:rho} multiplied by $\bu$, we have instead of \eqref{pseudo:mom:bu}:
\begin{equation}\label{mom:alt:2}
\begin{aligned}
& \pd_t (\rho \bu) + \div (\rho \bu \otimes \bu) \\
& \quad = - \nabla \tilde{p} + \mu \nabla \phi +  \div (2 \eta(\phi) \D \bu + 2 \eta_r(\phi) \W \bu) \\
& \qquad + 2 \curl(\eta_r(\phi) \bo) + \div (\rho'(\phi) m(\phi) \nabla \mu \otimes \bu )+ \frac{1}{2} R\bu.
\end{aligned}
\end{equation}
In a similar fashion, adding \eqref{PF:ang:mom} with \eqref{PF:rho} multiplied by $j^*\bo$, we have
\begin{equation}\label{ang:mom:alt:2}
\begin{aligned}
& \pd_t (\rho j^* \bo) + \div (\rho \bu \otimes j^* \bo ) \\
&  \quad = \div (c_0(\phi) (\div \bo) \I + 2c_d(\phi) \D \bo +2 c_a(\phi) \W \bo)   \\
& \qquad  + 2 \eta_r(\phi)(\curl \bu - 2 \bo) + \div (\rho'(\phi) m(\phi) \nabla \mu \otimes \bo) +  \frac{j^*}{2} R \bo.
\end{aligned}
\end{equation}
\end{remark}

\subsection{Model with nonlocal Ginzburg--Landau energy}\label{sec:nonlocal}
Instead of \eqref{GL:eg} we consider the nonlocal energy \eqref{nlGL} with a suitable symmetric convolution kernel $K$.  Then, 
\begin{align*}
\nd (\widehat{\rho \psi}) & = \frac{\sigma}{\eps} F'(\phi)  \nd \phi + \frac{\delta}{\delta \phi} \big ( \frac{\sigma \eps}{4} \int_\Omega K(x-y)|\phi(x) - \phi(y)|^2 \, dy \big ) \nd \phi \\
& = \frac{\sigma}{\eps} F'(\phi) \nd \phi + \sigma \eps \big ( (K  \star 1) \phi - (K \star \phi) \big ) \nd \phi,
\end{align*}
where $\star$ denotes the convolution operation, i.e.,
\[
(K \star \zeta)(x) = \int_\Omega K(x-z) \zeta(z) \, dz.
\]
In comparison to \eqref{md:free} we note that the symmetric tensor $\sigma \eps (A'(\nabla \phi) \otimes \nabla \phi)$ related to the capillary stress does not appear. Then, the analogue of \eqref{bu:bulk:diss} reads as
\begin{equation}\label{bu:bulk:diss:nl}
\begin{aligned}
- \mathcal{D} & = - \nabla \bu : \Big ( \bm{T} + \bm{J} \otimes \bu \Big ) - \nabla \bo : \Big ( \bm{C} + \bm{J} \otimes  j^*\bo \Big ) \\
& \quad + \Big ( \frac{\sigma}{\eps} F'(\phi)  + \sigma \eps \big ( (K \star 1) \phi - (K \star \phi) \big ) - \mu \Big ) \nd{\phi}  -\bm{\xi} \cdot \nabla \nd{\phi} \\
& \quad  + \nabla \mu \cdot \bm{J}_\phi + \big ( \bm{k} - \frac{1}{2} R\bu \big ) \cdot \bu + \big ( \bm{T}_x + \bm{g} \times \bu + \bm{h} - \frac{1}{2} j^* R \bo \big ) \cdot \bo \\
& =  - \D \bu : \bm{S}_s - \bm{S}_a : \Big (\W \bu - \bm{\Omega}\Big )  - \D \bo : \bm{E}_s - \W \bo : \bm{E}_a  + \bo \cdot ((\bm{g} - \bm{J}) \times \bu) \\
& \quad + \Big ( \frac{\sigma}{\eps} F'(\phi) + \sigma \eps \big ( (K \star 1) \phi - (K \star \phi) \big ) - \mu \Big ) \nd{\phi} - \bm{\xi} \cdot \nabla \nd{\phi} + \nabla \mu \cdot \bm{J}_\phi  \\
& \quad + \big ( \bm{k} - \frac{1}{2} R\bu \big ) \cdot \bu + \big (\bm{h} - \frac{1}{2} j^* R \bo \big ) \cdot \bo \\
& \leq 0.
\end{aligned}
\end{equation}
after employing \eqref{skew:ten} and \eqref{skew:rel}. By a Coleman--Noll procedure, we consider the constitutive assumptions
\begin{align}\label{bu:nl:constit}
\bm{\xi} = \bm{0}, \quad \mu = \frac{\sigma}{\eps} F'(\phi) + \sigma \eps \big ( (K \star 1) \phi - (K \star \phi) \big ), \quad \bm{S}_s  =  2\eta(\phi) \D \bu,
\end{align}
with $\bm{J}_\phi$, $\bm{g}$, $\bm{S}_a$, $\bm{E}_s$, $\bm{E}_a$, $\bm{h}$ and $\bm{k}$ chosen as in \eqref{bu:constit}. Keeping $\chi$ unchanged as in \eqref{chi:defn}, we find that the analogue of \eqref{bu:surf:diss} now reads as
\begin{equation}\label{bu:surf:diss:nl}
\begin{aligned}
- \mathcal{D}_\Gamma & := \Big ( \hat{G}'(\phi) - \beta \divS (B'(\nS \phi))  \Big ) \nd \phi + \divS \Big ( \hat{\chi} \bu + \beta \pd_t \phi B'(\nS \phi) - \bm{t} \Big ) \\
& \quad +  \Big ( \bm{\nu} \cdot (\bm{T} - p\I + \bm{J} \otimes \bu) - \big ( \hat{G}'(\phi) - \beta \divS (B'(\nS \phi)) \big ) \nS \phi \Big ) \cdot \bu \\
& \quad + [\bm{\nu} \cdot (\bm{C} + \bm{J} \otimes j^* \bo)] \cdot \bo - \mu \bm{J}_\phi \cdot \bm{\nu} - \frac{1}{2}(|\bu|^2 + j^* |\bo|^2) \bm{J} \cdot \bm{\nu} \leq 0.
\end{aligned}
\end{equation}
This leads to the boundary conditions
\begin{align}\label{phi:bc:nl}
\nd \phi = - \gamma(\phi) \Big ( \hat{G}'(\phi) - \beta \divS (B'(\nS \phi))\Big ), \quad m(\phi) \pdnu \mu = 0,
\end{align}
where we note the minor difference with \eqref{phi:bc}, along with \eqref{no:pene}, \eqref{u:bc} and \eqref{w:bc} for the inclusion of moving contact line dynamics. Hence, a nonlocal analogue of \eqref{bu:PF:model}-\eqref{bu:PF:GNBC} reads as
\begin{subequations}\label{bu:PF:model:nl}
\begin{alignat}{2}
\div \bu & = 0, \\[1ex]
\nd{\phi} & = \div (m(\phi) \nabla \mu), \\[1ex]
 \mu & = \frac{\sigma}{\eps} F'(\phi) + \sigma \eps \big ( (K \star 1) \phi - (K \star \phi) \big ), \\[1ex] 
 \nd \rho & = \div (\rho'(\phi) m(\phi) \nabla \mu)  + R, \quad R = - m(\phi) \nabla \mu \cdot \nabla \rho'(\phi),  \label{PF:rho:equ:R} \\[1ex]
 \rho \nd{u} &= - \nabla p + \div (2 \eta(\phi) \D \bu + 2 \eta_r(\phi) \W \bu) \label{lin:mom:nl} \\[1ex] 
\notag & \quad + 2 \curl(\eta_r(\phi) \bo) + \rho'(\phi) \big ( m(\phi) \nabla \mu \cdot \nabla \big ) \bu - \frac{1}{2} R \bu, \\[1ex]
\rho j^* \nd{\bo} & = \div (c_0 (\div \bo) \I + 2c_d \D \bo +2 c_a \W \bo) + 2 \eta_r(\phi)(\curl \bv - 2 \bo)  \\[1ex]
 & \quad + j^* \rho'(\phi) \big ( m(\phi) \nabla \mu \cdot \nabla \big ) \bo - \frac{j^*}{2} R \bo, 
\end{alignat}
\end{subequations}
with boundary conditions
\begin{subequations}\label{bu:PF:GNBC:nl}
\begin{alignat}{2}
& m(\phi) \pdnu \mu = 0, \quad \bu \cdot \bnu = 0, \quad \bo \cdot \bnu = 0  \\[1ex]
& \nd{\phi} = - \gamma(\phi) \big (\hat{G}'(\phi) - \beta \divS(B'(\nS \phi)) \big ), \\[1ex]
& c_u(\phi) \bu_\tau +  \Big ( 2 \bnu \cdot [\eta(\phi) \D \bu + \eta_r(\phi) \W \bu - \eta_r(\phi) \bm{\Omega}]  \Big )_\tau = -\frac{1}{\gamma(\phi)} \nd \phi \nS \phi, \\[1ex]
& c_w (\phi) \bo_\tau + \Big ( 2 \bnu \cdot [c_d(\phi) \D \bo + c_a(\phi) \W \bo] \Big )_\tau  = \bm{0}.
\end{alignat}
\end{subequations}

\begin{remark}[Reformulation of the linear pseudo-momentum balance]\label{rem:mom:nl}
Using the identity
\begin{align*}
\mu \nabla \phi & = \frac{\sigma}{\eps} F'(\phi) \nabla \phi + \sigma \eps \Big ( (K \star 1) \phi - (K \star \phi) \Big ) \nabla \phi \\
& = \nabla \Big ( \frac{\sigma}{\eps} F(\phi) + \frac{\sigma \eps}{4} \int_\Omega K(x-y) (\phi(x) - \phi(y))^2 \, dy \Big ),
\end{align*}
we can introduce a rescaled pressure $\hat{p} := p + \frac{\sigma}{\eps} F(\phi) + \frac{\sigma \eps}{4} \int_\Omega K(x-y) (\phi(x) - \phi(y))^2 \, dy$ to express
the balance of linear pseudo-momentum \eqref{lin:mom:nl} in a form analogous to \eqref{mom:alt}:
\begin{equation}\label{mom:alt:nl}
\begin{aligned}
 \rho \nd{u} & = - \nabla \hat{p} + \mu \nabla \phi + \div (2 \eta(\phi) \D \bu + 2 \eta_r(\phi) \W \bu) \\
  & \quad + 2 \curl(\eta_r(\phi) \bo) + \rho'(\phi) \big ( m(\phi) \nabla \mu \cdot \nabla \big ) \bu - \frac{1}{2} R \bu.
\end{aligned}
\end{equation}
This form of the balance of linear pseudo-momentum incorporates the capillary stresses in the term $\mu \nabla \phi$.
\end{remark}

\section{Existence of weak solutions}\label{sec:analysis}
In this section we establish the existence of weak solutions to the new phase field model \eqref{bu:PF:model}-\eqref{bu:PF:GNBC}. Our approach follows the work of \cite{GGW}, which treats the case without micro-rotations.  Without loss of generality and to simplify the presentation, let us set constants and parameters $\sigma$, $\eps$, $j^*$, $\beta$, $\gamma$, $c_u$ and $c_w$ to be $1$, letting $A(\bm{p})=  \frac{1}{2} |\bm{p}|^2$ and $B(\bm{p}) = \frac{1}{2} |\bm{p}|^2$ to be the isotropic gradient energies, so that $A'(\nabla \phi) = \nabla \phi$ and $B'(\nS \phi) = \nS \phi$.  Furthermore, we consider a specific decomposition of the interfacial free energy density $\hat{G}$ into
\[
\hat{G}(s) = \frac{\zeta}{2}|s|^2 + G(s),
\]
with constant $\zeta >0$. Then, the resulting simplified model that satisfies the physically relevant constraint $\phi \in [-1,1]$ reads in strong formulation as
\begin{subequations}\label{ana:bulk}
\begin{alignat}{2}
\div \bu & = 0, \\[1ex]
\nd{\phi} & = \div (m(\phi) \nabla \mu), \label{simp:phi} \\[1ex]
 \mu & = F'(\phi) - \Delta \phi, \label{simp:mu} \\[1ex] 
\label{simp:bu} \rho \nd{\bu} & = - \nabla p + \mu \nabla \phi + \div (2 \eta(\phi) \D \bu + 2 \eta_r(\phi) \W \bu) \\[1ex] 
\notag & \quad + 2 \curl(\eta_r(\phi) \bo) + \tfrac{\ov{\rho}_1 - \ov{\rho}_2}{2} \big ( m(\phi) \nabla \mu \cdot \nabla \big ) \bu, \\[1ex]
\label{simp:bo} \rho \nd{\bo} & = \div (c_0(\phi) (\div \bo) \I + 2c_d(\phi) \D \bo +2 c_a(\phi) \W \bo)   \\[1ex]
 & \notag \quad + 2 \eta_r(\phi)(\curl \bu - 2 \bo) + \tfrac{\ov{\rho}_1 - \ov{\rho}_2}{2} \big ( m(\phi) \nabla \mu \cdot \nabla \big ) \bo, 
\end{alignat}
\end{subequations}
with boundary conditions
\begin{subequations}\label{ana:bc}
\begin{alignat}{2}
& m(\phi) \pdnu \mu = 0, \quad \bu \cdot \bnu = 0, \quad \bo \cdot \bnu = 0  \\[1ex]
& \nd{\phi} = - \big ( \zeta \phi + G'(\phi) -  \Delta_\Gamma \phi + \pdnu \phi \big ) =: - \mathcal{L}(\phi), \\[1ex]
& \bu_\tau +  \Big ( 2 \bnu \cdot [\eta(\phi) \D \bu + \eta_r(\phi) \W \bu - \eta_r(\phi) \bm{\Omega}]  \Big )_\tau = -\nd \phi \nS \phi, \\[1ex]
& \bo_\tau + \Big ( 2 \bnu \cdot [c_d(\phi) \D \bo + c_a(\phi) \W \bo] \Big )_\tau  = \bm{0},
\end{alignat}
\end{subequations}
where with a slight abuse of notation we use the symbol $p$ for the rescaled pressure in \eqref{mom:alt}. Furthermore, to establish the existence of a suitable notion of weak solutions for \eqref{ana:bulk}-\eqref{ana:bc}, in the following we perform a two level approximation with parameters which, with an abuse of notation, involve symbols such as $\sigma$ and $\eps$ that bear no relation to the physical parameters of surface tension and interfacial thickness in \eqref{GL}.

\subsection{Preliminaries}
Let $X$ be a real Banach space and $X^*$ be its topological dual, so that $\inn{f}{g}_X$ for $f \in X^*$ and $g \in X$ is the corresponding duality pairing. We write $X \subset Y$ and $X \Subset Y$ to denote the continuous and compact embedding of $X$ into $Y$. 

For $1 \leq p \leq \infty$, the Bochner space $L^p(0,T;X)$ denotes the set of all strongly measurable $p$-integrable functions (if $p < \infty$) or essentially bounded functions (if $p = \infty$) on the time interval $[0,T]$ with values in the Banach space $X$. The space $W^{1,p}(0,T;X)$ denotes all $u \in L^p(0,T;X)$ such that its vector-valued distributional derivative $\frac{d u}{dt} \in L^p(0,T;X)$. Furthermore, $C^0([0,T];X)$ denotes the Banach space of all bounded and continuous functions $u:[0,T] \to X$ equipped with the supremum norm, while $C_w^0([0,T];X)$ is the space of bounded and weakly continuous functions $u:[0,T] \to X$, i.e., $\inn{f}{u} : [0,T] \to \R$ is continuous for all $f \in X^*$. Furthermore, the notation $C^\infty_0(0,T;X)$ denotes the vector space of all smooth functions $u:(0,T) \to X$ with $\text{supp}(u) \Subset (0,T)$.

For a bounded domain $\Omega \subset \R^d$, $d \in \{2,3\}$ with smooth boundary $\Gamma := \pd \Omega$, we denote by $L^p(\Omega)$ and $W^{k,p}(\Omega)$, for $1 \leq p \leq \infty$ and $k \geq 0$ to be the Lebesgue and Sobolev spaces over $\Omega$. Similarly, $L^p(\Gamma)$ and $W^{k,p}(\Gamma)$ are the Lebesgue and Sobolev spaces over $\Gamma$.  When $k \in \mathbb{N}$ and $ p = 2$, we use the short hand $H^k(\Omega) = W^{k,2}(\Omega)$ and $H^k(\Gamma) = W^{k,2}(\Gamma)$.  We use $(\cdot,\cdot)_\Omega$ to denote the $L^2(\Omega)$ inner product with associated norm $\| \cdot \|$, and $(\cdot,\cdot)_\Gamma$ denotes the $L^2(\Gamma)$ inner product with associated norm $\| \cdot \|_\Gamma$.  The norms of $W^{k,p}(\Omega)$ and $W^{k,p}(\Gamma)$ will be denoted by $\| \cdot \|_{W^{k,p}}$ and $\| \cdot \|_{W^{k,p}_\Gamma}$, respectively.  In particular we will use the space
\[
H^2_n(\Omega) := \{ f \in H^2(\Omega) \, : \, \pdnu f = 0 \text{ on } \Gamma \}.
\]
We recall that for $p \in (1,\infty)$, $k > \frac{1}{p}$ with $k \in \mathbb{N}$ the trace operator 
\[
\mathrm{tr} : W^{k,p}(\Omega) \to W^{k-\frac{1}{p},p}(\Gamma), \quad \mathrm{tr}(u) = u \vert_\Gamma,
\]
is a continuous operator, where for the case $p \neq 2$ additionally we require $k - \frac{1}{p} \notin \mathbb{N}$. It is convenient to treat the boundary trace of $\phi$ as a separate variable $\psi$, which leads to the consideration of the function space
\begin{align*}
\VV^k := \Big \{ (\phi, \psi) \in H^k(\Omega) \times H^{k - \frac{1}{2}}(\Gamma) \, : \, \psi = \mathrm{tr}(\phi) \in H^k(\Gamma) \Big \}, 
\end{align*}
for $k \in \mathbb{N}$, equipped with the inner product and norm
\[
\inn{(\phi, \psi)}{(\hat{\phi}, \hat{\psi})}_{\VV^k} := (\phi, \hat{\phi})_{H^k} + (\psi, \hat{\psi})_{H^k_\Gamma}, \quad \| (\phi, \psi) \|_{\VV^k} = \sqrt{\| \phi \|_{H^k}^2 + \| \psi \|_{H^k(\Gamma)}^2}.
\]
For convenience we use the notation $Q := \Omega \times (0,T)$ and $\Sigma := \Gamma \times (0,T)$, as well as $(\cdot,\cdot)_Q$ and $(\cdot, \cdot)_\Sigma$ to denote the $L^2$-inner product on $Q$ and $\Sigma$, respectively.

For the velocity field and the micro-rotation field, we consider the following functional framework. We denote by 
\begin{align*}
\C^\infty_{*}(\overline{\Omega}) & := \big \{ \bm{f} \in C^\infty(\overline{\Omega}; \R^d) \, : \, \bm{f} \cdot \bm{\nu} = 0 \text{ on } \Gamma \big \}, \\
\C^{\infty}_{\div}(\overline{\Omega}) & := \big \{ \bm{f} \in \C^\infty_*(\overline{\Omega}) \, : \,  \div \bm{f} = 0 \text{ in } \Omega \big \},
\end{align*}
and set for any $s \geq 0$ and $p \in [1,\infty]$,
\[
\WW^{s,p} := \overline{\C^\infty_*(\overline{\Omega})}^{W^{s,p}(\Omega)}, \quad \HH^s := \WW^{s,2}, \quad \WW^{s,p}_{\div} := \overline{\C^\infty_{\div}(\overline{\Omega})}^{W^{s,p}(\Omega)}, \quad \HH^s_{\div} := \WW^{s,2}_{\div}.
\]
We use the notation $\HH = \HH^0$ and $\HH_{\div} = \HH_{\div}^0$ for the case $s = 0$ where we identify $W^{0,2}(\Omega) = L^2(\Omega)$. 
\subsection{Assumptions and main result}
We make the following set of assumptions (see also \cite{GGM,GGW}):
\begin{assump}\label{ass:main}
\
\begin{enumerate}[label=$(\mathrm{A \arabic*})$, ref = $\mathrm{A \arabic*}$]
\item \label{ass:dom} $\Omega \subset \R^d$, $d \in \{2, 3\}$, is a bounded domain with smooth boundary $\Gamma$. 
\item \label{ass:coeff} We assume that $m, \eta_0, \eta_r, c_0, c_d, c_a \in C^{1,1}_{\mathrm{loc}}(\R)$ and there exist positive constants $k_0 < k_1$ such that 
\[
0 < k_0 \leq m(s), \eta_0(s), \eta_r(s), c_0(s), c_d(s), c_a(s) \leq k_1.
\]
\item \label{ass:G} There exist constants $C_G > 0$ and $c_G \geq 0$ such that the free energy density $G$ satisfies the properties: $G \in C^2(\R)$ can be decomposed into a sum $G = G_0 + G_1$ where $G_0 \in C^2(\R)$ is convex and $G_1 \in C^2(\R)$ is concave, along with 
\[
G(s) \geq - c_G, \quad G''(s) \geq - c_G, \quad |G''(s)| \leq C_G(1 + |s|^p) \quad \forall s \in \R,
\]
for arbitrary but fixed $p \in [1,\infty)$.
\item \label{ass:rho} The mass density function $\rho$ is given as
\[
\rho(s) = \frac{\overline{\rho}_1 - \overline{\rho}_2}{2} s + \frac{\overline{\rho}_1 + \overline{\rho}_2}{2} \quad \forall s \in [-1,1],
\] 
with fixed positive constant mass densities $\overline{\rho}_1$ and $\overline{\rho}_2$.
\item \label{ass:F} The free energy density $F$ satisfies the properties:
\begin{itemize}
\item $F(s) = F_0(s) + F_1(s)$ where $F_0 \in C^0([-1,1]) \cap C^2((-1,1))$, $F_1(s) \in C^2(\R)$ with $F_0(0) = 0$, $F_0'(0) = 0$, $F_0''(s) \geq 0$ for $s \in (-1,1)$ and
\[
\lim_{s \to \pm 1} F_0'(s) = \pm \infty, \quad \lim_{s \to \pm 1} F_0''(s) = +\infty, \quad |F_1''(r)| \leq C_F \quad \forall r \in \R,
\]
so that $F''(s) \geq - C_F$ for all $s \in (-1,1)$ for some constant $C_F > 0$.
\item There exist constants $M \in (0,1)$, $\delta > 0$, $C_{\delta, M} > 0$ and $C_M > 0$ such that for all $s \in (-1,M] \cup [M,1)$,
\[
F_0''(s) - \delta (F_0'(s))^2 \geq - C_{\delta, M}, \quad F_0'(s) (\zeta s + G'(s)) \geq - C_M.
\]
\end{itemize}
\item \label{ass:ini} $\bu_0 \in \HH_{\div}$, $\bo_0 \in \HH$, $(\phi_0,\mathrm{tr}(\phi_0)) \in \VV^1$ with $|\phi_0| \leq 1$ a.e.~in $\Omega$, $F(\phi_0) \in L^1(\Omega)$, $F(\mathrm{tr}(\phi_0)) \in L^1(\Gamma)$ and $\frac{1}{|\Omega|} \int_\Omega \phi_0 \, dx \in (-1,1)$.
\end{enumerate}
\end{assump}

We remark that \eqref{ass:G} and \eqref{ass:F} are fulfilled with the classical logarithmic potential \eqref{log:pot} for $F$:
\begin{align}\label{log}
F_0(s) = (1+s) \ln(1+s) + (1-s) \ln (1-s), \quad F_1(s) = \frac{-c_F}{2} s^2,
\end{align}
and for any function $\hat{G}(s) = \frac{\zeta}{2}|s|^2 + G(s)$ that fulfills the sign condition $\pm \hat{G}'(\pm 1) > 0$, such as the second example in \eqref{hatG:eg} with $\gamma_1 > \gamma_2$. We now introduce the notion of weak solutions, which treats the trace of the order parameter $\phi \vert_\Gamma$ as a new unknown variable $\psi$ defined on the boundary:
\begin{defn}\label{defn:weaksoln}
A quintuple of functions $(\bu, \bo, \mu, (\phi, \psi))$ is a weak solution to \eqref{ana:bulk}-\eqref{ana:bc} on the time interval $[0,T]$ if the following properties are satisfied:
\begin{itemize}
\item Regularity
\begin{equation*}
\begin{aligned}
\bu & \in C_w([0,T]; \HH_{\div}) \cap L^2(0,T;\HH^1), \\
\bo & \in C_w([0,T]; \HH) \cap L^2(0,T;\HH^1), \\
(\phi, \psi) & \in C_w([0,T];\VV^1) \cap L^2(0,T;\VV^2), \\
\mu & \in L^2(0,T;H^1(\Omega)),
\end{aligned}
\end{equation*}
along with $|\phi| < 1$ a.e.~in $Q$ and $|\psi| \leq 1$ a.e.~on $\Sigma$.
\item Equations
\begin{subequations}
\begin{alignat}{2}
\label{weak:ns} 0 &=  -(\rho(\phi) \bu, \pd_t \bm{v})_Q + (\div(\rho \bu \otimes \bu), \bm{v})_Q  \\
\notag & \quad + (2 \eta(\phi) \D \bu, \D \bm{v})_Q + (2 \eta_r(\phi) \W \bu, \W \bm{v})_Q - (2\eta_r(\phi) \bo, \curl \bm{v})_Q  \\
\notag & \quad  + (\bu_\tau, \bm{v}_\tau)_\Sigma - (\mathcal{L}(\psi) \nS \psi, \bm{v}_\tau)_{\Sigma} - (( \bm{J} \otimes \bu), \nabla \bm{v})_Q - (\mu \nabla \phi, \bm{v})_Q \\[1ex]
\label{weak:w} 0 & = -(\rho(\phi) \bo, \pd_t \bm{z})_Q + (\div (\rho \bu \otimes \bo), \bm{z})_Q + (\bo_\tau, \bm{z}_\tau)_{\Sigma}  \\
\notag & \quad + (c_0(\phi) \div \bo, \div \bm{z})_Q  + (2 c_d(\phi) \D \bo, \D \bm{z})_Q + (2 c_a(\phi) \W \bo, \W \bm{z})_Q\\
\notag & \quad  - (2 \eta_r(\phi) (\curl \bu - 2 \bo), \bm{z})_Q - (( \bm{J} \otimes \bo ), \nabla \bm{z})_Q, \\[1ex]
\label{weak:phi} 0 & = -(\phi, \pd_t \xi)_Q + (\bu \cdot \nabla \phi, \xi)_Q + (m(\phi) \nabla \mu, \nabla \xi)_Q, \\[1ex]
\label{weak:psi} 0 & = -(\psi, \pd_t \Upsilon)_\Sigma + (\bu_\tau \cdot \nS \psi, \Upsilon)_\Sigma + ( \mathcal{L}(\psi), \Upsilon)_\Sigma,
\end{alignat}
\end{subequations}
holding for all $\bm{v} \in C^\infty_0(0,T;\C^\infty_{\div}(\overline{\Omega}))$, $\bm{z} \in C^\infty_0(0,T;\C^\infty_*(\overline{\Omega}))$, $\xi \in C^\infty_0(0,T;C^1(\overline{\Omega}))$ and $\Upsilon \in C^\infty_0(0,T;C^0(\Gamma))$, along with 
\begin{subequations}
\begin{alignat}{2}
\label{weak:mu} \mu & = - \Delta \phi + F'(\phi) && \quad \text{ a.e.~in } Q, \\[1ex]
\label{weak:J} \bm{J} & = -\tfrac{\overline{\rho}_1 - \overline{\rho}_2}{2} m(\phi) \nabla \mu && \quad \text{ a.e.~in } Q, \\[1ex]
\label{weak:rho} \rho(\phi) & = \tfrac{\overline{\rho}_1 - \overline{\rho}_2}{2} \phi + \tfrac{\overline{\rho}_1+ \overline{\rho}_2}{2} && \quad  \text{ a.e.~in } Q, \\[1ex]
\label{weak:L} \mathcal{L}(\psi) & = - \Delta_\Gamma \psi + \pdnu \phi + \zeta \psi + G'(\psi) && \quad \text{ a.e.~on } \Sigma.
\end{alignat}
\end{subequations}
\item Energy inequality
\begin{equation}\label{energy:ineq}
\begin{aligned}
& E(\bu(t), \bo(t), \phi(t), \psi(t)) \\
& \qquad + \int_s^t \int_\Omega m(\phi) |\nabla \mu|^2 + 2 \eta(\phi) |\D \bu|^2 + 4 \eta_r(\phi) | \tfrac{1}{2} \curl \bu - \bo|^2 \, dx \, d\tau \\
& \qquad + \int_s^t \int_\Omega c_0(\phi) |\div \bo|^2 + 2 c_d(\phi) |\D \bo|^2 + 2 c_a(\phi) |\W \bo|^2 \, dx \, d \tau \\
& \qquad + \int_s^t \int_{\Gamma} |\bu_\tau|^2 + |\bo_\tau|^2 + |\mathcal{L}(\psi)|^2 \, dS \, d \tau \\
& \quad \leq E(\bu(s), \bo(s), \phi(s), \psi(s))
\end{aligned}
\end{equation}
holds for all $t \in [s,\infty)$ and almost all $s \in [0,\infty)$ including $s = 0$, where the total energy $E$ is given by
\begin{equation}\label{total:energy}
\begin{aligned}
E(\bu, \bo, \phi, \psi) :& = \int_\Omega \frac{\rho(\phi)}{2} |\bu|^2 + \frac{\rho(\phi)}{2} |\bo|^2 + \frac{1}{2}|\nabla \phi|^2 + F(\phi) \, dx \\
& \quad + \int_{\Gamma} \frac{1}{2} |\nS \psi|^2 + \frac{\zeta}{2} |\psi|^2 + G(\psi) \, dS
\end{aligned}
\end{equation}
\item Initial conditions
\[
(\bu, \bo, \phi, \psi) \vert_{t = 0} = (\bu_0, \bo_0, \phi_0, \mathrm{tr}(\phi_0)).
\]
\end{itemize}
\end{defn}

Our result concerning the existence of weak solutions to \eqref{ana:bulk}-\eqref{ana:bc} is formulated as follows:

\begin{thm}\label{thm:exist}
Under Assumption \ref{ass:main}, for any $T \in (0,\infty)$ there exists a weak solution to \eqref{ana:bulk}-\eqref{ana:bc} in the sense of Definition \ref{defn:weaksoln}.
\end{thm}
Note that uniqueness of weak solutions even in the matched density case $\overline{\rho}_1 = \overline{\rho}_2$ in two spatial dimensions remains an open problem, cf.~\cite{GGM}. The proof of Theorem~\ref{thm:exist} follows a similar structure to that of \cite{GGW}, which involves first studying an approximating problem with regular potentials and additional regularisation terms. Then, we obtain solutions of an associated problem with singular potentials, and finally solutions to the original model \eqref{ana:bulk}-\eqref{ana:bc}. 

\subsection{Approximate model with regular potentials and regularisation terms}
The standard approach is to approximate the singular logarithmic part $F_0$ by a sequence of smooth polynomial potentials. But this in turn leads to the loss of the boundedness property of $\phi$ in the interval $[-1,1]$, and consequently $\rho'(\phi) \neq \frac{\ov{\rho}_1 - \ov{\rho}_2}{2}$ if $\phi \notin [-1,1]$.  Hence, the appropriate model to consider is \eqref{bu:PF:model}-\eqref{bu:PF:GNBC} where the additional terms involving $R = -m(\phi) \nabla \mu \cdot \nabla \rho'(\phi)$ appear. Following \cite{GGW} we add non-Newtonian type regularizing terms to the linear pseudo-momentum balance \eqref{mom:alt:2} and the angular momentum balance \eqref{ang:mom:alt:2}, as well as a viscous regularisation term to the equation \eqref{simp:mu}.  For $\eps, \sigma, \kappa \in (0,1]$ and $q > 2d$, the approximate model expressed in strong formulation reads as
\begin{subequations}\label{app:model1}
\begin{alignat}{2}
\div \bu  & = 0, \\[1ex]
\pd_t \phi + \bu \cdot \nabla \phi & = \div (m(\phi) \nabla \mu), \label{app:model:phi} \\[1ex]
\mu & = - \Delta \phi + F_\kappa'(\phi) + \sigma \pd_t \phi, \label{app:model:mu} \\[1ex]
\pd_t (\rho \bu) + \div (\rho \bu \otimes \bu) & = - \nabla p + \mu \nabla \phi +  \div (2 \eta(\phi) \D \bu + 2 \eta_r(\phi) \W \bu) \label{app:model:u} \\[1ex]
\notag & \quad + 2 \curl(\eta_r(\phi) \bo) + \div (\rho'(\phi) m(\phi) \nabla \mu \otimes \bu )+ \frac{1}{2} R\bu \\[1ex]
\notag & \quad + \eps \Big ( \div ( |\nabla \bu|^{q-2} \nabla \bu) - |\bu|^{q-2} \bu  \Big ) \\[1ex]
\pd_t (\rho \bo) + \div (\rho \bu \otimes  \bo ) & = \div (c_0(\phi) (\div \bo) \I + 2c_d(\phi) \D \bo +2 c_a(\phi) \W \bo)  \label{app:model:w}  \\[1ex]
\notag & \quad  + 2 \eta_r(\phi)(\curl \bv - 2 \bo) + \div (\rho'(\phi) m(\phi) \nabla \mu \otimes \bo) \\
\notag & \quad +  \frac{1}{2} R \bo + \eps \Big ( \div (|\nabla \bo|^{q-2} \nabla \bo) - |\bo|^{q-2} \bo \Big ),
\end{alignat}
\end{subequations}
where $R = - m(\phi) \nabla \mu \cdot \nabla \rho'(\phi)$ and subjected to the modified boundary conditions (on account of the non-Netwonian regularisation terms)
\begin{subequations}\label{app:model1:bc}
\begin{alignat}{2}
& m(\phi) \pdnu \mu = 0, \quad \bu \cdot \bnu = 0, \quad \bo \cdot \bnu = 0  \\[1ex]
& \nd{\psi} = - \big ( \zeta \psi + g'(\psi) -  \Delta_\Gamma \psi + \pdnu \psi \big ) = - \mathcal{L}(\psi), \quad \psi = \phi \vert_{\Sigma}, \\[1ex]
& \bu_\tau +  \Big ( \bnu \cdot [\eps |\nabla \bu|^{q-2} \nabla \bu + 2\eta(\phi) \D \bu + 2\eta_r(\phi) \W \bu - 2\eta_r(\phi) \bm{\Omega}]  \Big )_\tau = \mathcal{L}(\psi) \nS \phi, \\[1ex]
& \bo_\tau + \Big (\bnu \cdot [ \eps |\nabla \bo|^{q-2} \nabla \bo + 2c_d(\phi) \D \bo + 2 c_a(\phi) \W \bo] \Big )_\tau  = \bm{0}.
\end{alignat}
\end{subequations}
In comparison with \cite{GGW}, our non-Newtonian regularisation terms contains the full gradient due to the presence of the anti-symmetric tensors $\W\bu$ and $\W \bo$. In \eqref{app:model:mu} the function $F_\kappa$ is defined as
\[
F_\kappa(s) = F_{0,\kappa}(s) + F_1(s),
\]
where we take $F_{0,\kappa}$ as a polynomial approximation of $F_0$.  For instance a quadratic approximation entails the definition
\[
F_{0,\kappa}''(s) = \begin{cases}
F_0''(1-\kappa) & \text{ for } s \geq 1- \kappa, \\
F_0''(s) & \text{ for } |s| \leq 1- \kappa, \\
F_0''(\kappa - 1) & \text{ for } s \leq -1 + \kappa, 
\end{cases} \quad \text{ with } F_{0,\kappa}(0) = F_{0,\kappa}'(0) = 0,
\] 
which further satisfies
\begin{align}\label{F0approx}
|F_{0,\kappa}'(s)| \leq |F_0'(s)|, \quad |F_{0,\kappa}(s)| \leq |F_0(s)| \quad \forall s \in (-1,1).
\end{align}
Then, the outline of the proof of Theorem \ref{thm:exist} can be roughly summarized as
\begin{itemize}
\item Step 1: Establish the existence of a weak solution to \eqref{app:model1}-\eqref{app:model1:bc} in the sense of Definition \ref{defn:app:solution} below.
\item Step 2: Pass to the limit $\kappa \to 0$ to obtain the singular logarithmic potential, along with the property that $|\phi| < 1$ a.e.~in $Q$ and $|\psi| \leq 1$ a.e.~on $\Sigma$. Consequently, the terms involving $R = - m(\phi) \nabla \mu \cdot \nabla \rho'(\phi)$ in \eqref{app:model:u} and \eqref{app:model:w} vanish.
\item Step 3: Pass to the limit $\eps \to 0$ and $\sigma \to 0$ to obtain a weak solution to \eqref{ana:bulk}-\eqref{ana:bc} in the sense of Definition \ref{defn:weaksoln}.
\end{itemize}
Since the equation for the micro-rotation $\bo$ is structurally similar to the equation for the velocity $\bu$, both variables share similar regularities in space and in time. Hence, the main deviation from \cite{GGW} lies in Step 1, which we will detail below.  The modifications to \cite{GGW} needed in Steps 2 and 3 to account for the micro-rotation are minimal and thus we omit the details.

To establish the existence of weak solutions for the approximate model \eqref{app:model1}-\eqref{app:model1:bc} it is sufficient to impose the following assumptions of $F_{\kappa}$ and on the density function $\rho$.

\begin{assump}\label{ass:main:approx}
\
\begin{enumerate}[label=$(\mathrm{B \arabic*})$, ref = $\mathrm{B \arabic*}$]
\item \label{ass:rho:approx} The mass density function $\rho \in C^2(\R)$ is bounded from below by a positive constant $\rho_0$ and satisfies $\rho, \rho', \rho'' \in L^\infty(\R)$.
\item \label{ass:F:approx} The approximate potential function $F_\kappa \in C^2(\R)$ satisfies
\[
F_\kappa(s) \geq \frac{\delta}{2} |s|^2 - c_F, \quad F_{\kappa}''(s) \geq - c_F, \quad |F_\kappa''(s)| \leq C_\kappa( 1 + |s|^p)
\]
for $p \in [1,\infty)$ if $d = 2$ and $p = 2$ if $d = 3$, where $\delta$ and $c_F$ are positive constant independent of $\kappa$, while $C_\kappa$ is a positive constant depending on $\kappa$.
\end{enumerate}
\end{assump}
We remark that \eqref{ass:rho:approx} does not conflict with equation \eqref{PF:rho:equ:R}, since multiplying \eqref{app:model:phi} with $\rho'(\phi)$ leads to
\[
\nd \rho(\phi) = \rho'(\phi) \pd_t \phi + \rho'(\phi) \nabla \phi \cdot \bu = \rho'(\phi) \div (m(\phi) \nabla \mu) = \div (\rho'(\phi) m(\phi) \nabla \mu) + R,
\]
which is \eqref{PF:rho:equ:R}. We now provide the definition of a weak solution to \eqref{app:model1}-\eqref{app:model1:bc} and the existence result.

\begin{defn}\label{defn:app:solution}
A quintuple of functions $(\bu, \bo, \mu, (\phi, \psi))$ is a weak solution to \eqref{app:model1}-\eqref{app:model1:bc} on the time interval $[0,T]$ if the following properties are satisfied:
\begin{itemize}
\item Regularity
\begin{equation*}
\begin{alignedat}{2}
\bu & \in C_w([0,T]; \HH_{\div}) \cap L^2(0,T;\HH^1) \cap L^q(0,T;\WW^{1,q}), \\
\bo &  \in C_w([0,T]; \HH) \cap L^2(0,T;\HH^1) \cap L^q(0,T;\WW^{1,q}), \\
(\phi, \psi) & \in C_w([0,T];\VV^1) \cap L^2(0,T;\VV^2), \quad \pd_t \phi \in L^2(0,T;L^2(\Omega)), \\
\mu & \in L^2(0,T;H^1(\Omega)).
\end{alignedat}
\end{equation*}
\item Equations
\begin{subequations}\label{equ:app:1}
\begin{alignat}{2}
\label{weak:app:ns} 0 &=  -(\rho(\phi) \bu, \pd_t \bm{v})_Q + (\div(\rho \bu \otimes \bu), \bm{v})_Q  \\
\notag & \quad + (2 \eta(\phi) \D \bu, \D \bm{v})_Q + (2 \eta_r(\phi) \W \bu, \W \bm{v})_Q - (2\eta_r(\phi) \bo, \curl \bm{v})_Q  \\
\notag & \quad  + (\bu_\tau, \bm{v}_\tau)_\Sigma - (\mathcal{L}(\psi) \nS \psi, \bm{v}_\tau)_{\Sigma} - (( \bm{J} \otimes \bu ), \nabla \bm{v})_Q - (\mu \nabla \phi, \bm{v})_Q \\
\notag & \quad + (\eps |\nabla \bu|^{q-2} \nabla \bu, \nabla \bm{v})_Q + (\eps |\bu|^{q-2} \bu, \bm{v})_Q - \tfrac{1}{2}(R \bu, \bm{v})_Q, \\[1ex]
\label{weak:app:w} 0 & = -(\rho(\phi) \bo, \pd_t \bm{z})_Q + (\div (\rho \bu \otimes \bo), \bm{z})_Q + (\bo_\tau, \bm{z}_\tau)_{\Sigma}  \\
\notag & \quad + (c_0(\phi) \div \bo, \div \bm{z})_Q  + (2 c_d(\phi) \D \bo, \D \bm{z})_Q + (2 c_a(\phi) \W \bo, \W \bm{z})_Q\\
\notag & \quad  - (2 \eta_r(\phi) (\curl \bu - 2 \bo), \bm{z})_Q - ((\bm{J} \otimes \bo), \nabla \bm{z})_Q, \\
\notag & \quad + (\eps |\nabla \bo|^{q-2} \nabla \bo, \nabla \bm{z})_Q + (\eps |\bo|^{q-2} \bo, \bm{z})_Q - \tfrac{1}{2}(R \bo, \bm{z})_Q,\\[1ex]
\label{weak:app:phi} 0 & = -(\phi, \pd_t \xi)_Q + (\bu \cdot \nabla \phi, \xi)_Q + (m(\phi) \nabla \mu, \nabla \xi)_Q, \\[1ex]
\label{weak:app:psi} 0 & = -(\psi, \pd_t \Upsilon)_\Sigma + (\bu_\tau \cdot \nS \psi, \Upsilon)_\Sigma + ( \mathcal{L}(\psi), \Upsilon)_\Sigma, \\
\label{weak:app:R:u} 0 & = (R\bu, \bm{v}) + (m(\phi) (\nabla \mu \cdot \nabla \rho'(\phi))\bu, \bm{v})_Q, \\
\label{weak:app:R:w} 0 & = (R\bo, \bm{z}) + (m(\phi) (\nabla \mu \cdot \nabla \rho'(\phi))\bo, \bm{z})_Q,
\end{alignat}
\end{subequations}
holding for all $\bm{v} \in C^\infty_0(0,T;\C^\infty_{\div}(\overline{\Omega}))$, $\bm{z} \in C^\infty_0(0,T;\C^\infty_*(\overline{\Omega}))$, $\xi \in C^\infty_0(0,T;C^1(\overline{\Omega}))$ and $\Upsilon \in C^\infty_0(0,T;C^0(\Gamma))$, along with 
\begin{subequations}\label{equ:app:2}
\begin{alignat}{2}
\label{weak:mu:approx} \mu & = - \Delta \phi + F_\kappa'(\phi) + \sigma \pd_t \phi && \quad \text{ a.e.~in } Q, \\[1ex]
\label{weak:J:approx} \bm{J} & = -\rho'(\phi) m(\phi) \nabla \mu && \quad \text{ a.e.~in } Q, \\[1ex]
\label{weak:L:approx} \mathcal{L}(\psi) & = - \Delta_\Gamma \psi + \pdnu \phi + \zeta \psi + G'(\psi) && \quad \text{ a.e.~on } \Sigma
\end{alignat}
\end{subequations}
\item Energy inequality
\begin{equation}\label{energy:ineq:approx}
\begin{aligned}
& E_\kappa(\bu(t), \bo(t), \phi(t), \psi(t)) \\
& \qquad + \int_s^t \int_\Omega m(\phi) |\nabla \mu|^2 + 2 \eta(\phi) |\D \bu|^2 + 4 \eta_r(\phi) | \tfrac{1}{2} \curl \bu - \bo|^2 \, dx \, d\tau \\
& \qquad + \int_s^t \int_\Omega c_0(\phi) |\div \bo|^2 + 2 c_d(\phi) |\D \bo|^2 + 2 c_a(\phi) |\W \bo|^2 \, dx \, d \tau \\
& \qquad + \int_s^t \int_{\Gamma} |\bu_\tau|^2 + |\bo_\tau|^2 + |\mathcal{L}(\psi)|^2 \, dS \, d \tau \\
& \qquad + \int_s^t \int_\Omega \sigma |\pd_t \phi|^2 + \eps \big (|\nabla \bu|^{q} + |\bu|^q + |\nabla \bo|^q + |\bo|^q \big ) \, dx \, d \tau \\
& \quad \leq E_\kappa(\bu(s), \bo(s), \phi(s), \psi(s))
\end{aligned}
\end{equation}
holds for all $t \in [s,\infty)$ and almost all $s \in [0,\infty)$ including $s = 0$, where the total energy $E_\kappa$ is given by \eqref{total:energy} with $F$ replaced by $F_\kappa$.
\item Initial conditions
\[
(\bu, \bo, \phi, \psi) \vert_{t = 0} = (\bu_0, \bo_0, \phi_0, \mathrm{tr}(\phi_0)).
\]
\end{itemize}
\end{defn}

\begin{thm}\label{thm:appmodel}
Under \eqref{ass:dom}, \eqref{ass:coeff}, \eqref{ass:G}, \eqref{ass:rho:approx} and \eqref{ass:F:approx}, for any $T \in (0,\infty)$, $\eps, \sigma, \kappa \in (0,1]$, and initial conditions $\bu_0 \in \HH_{\div}$, $\bo_0 \in \HH$, $(\phi_0, \psi_0) \in \VV^1$, there exists a global weak solution $(\bu_{\sigma,\eps,\kappa}, \bo_{\sigma,\eps,\kappa}, \mu_{\sigma, \eps,\kappa}, (\phi_{\sigma,\eps,\kappa}, \psi_{\sigma,\eps,\kappa}))$ to the approximate model \eqref{app:model1}-\eqref{app:model1:bc} in the sense of Definition \ref{defn:app:solution}. 
\end{thm}

The proof of Theorem \ref{thm:appmodel} is based on an implicit time discretisation, similarly used in \cite{ADG,GGW}. We sketch the details and remark on the modifications needed to extend the micropolar case.

\subsection{Implicit time discretisation}
For $N \in \mathbb{N}$ we define the time step $h = \frac{1}{N}$. Then, given $\bu_k \in \HH_{\div}$, $\bo_k \in \HH$, $(\phi_k, \psi_k) \in \VV^1$ with $F_\kappa'(\phi_k) \in L^2(\Omega)$, $G'(\psi_k) \in L^2(\Gamma)$ and $\rho_k = \rho(\phi_k)$, we construct the quintuple of functions:
\[
(\bu, \bo,  \mu, (\phi, \psi)) := (\bu_{k+1}, \bo_{k+1},  \mu_{k+1}, (\phi_{k+1}, \psi_{k+1}))
\] 
with 
\begin{align}\label{disc:J}
\bm{J} := \bm{J}_{k+1} := - \rho'(\phi_k) m(\phi_k) \nabla \mu_{k+1}
\end{align}
that satisfies the equations
\begin{subequations}\label{Time:dis}
\begin{alignat}{2}
\label{Time:u} 0& = \frac{1}{h}(\rho \bu - \rho_k \bu_k, \bm{v})_\Omega + (\div (\rho_k \bu \otimes \bu), \bm{v})_\Omega \\
\notag & \quad + (2 \eta(\phi_k) \D \bu, \D \bm{v})_\Omega + (2 \eta_r(\phi_k) \W \bu, \W \bm{v})_\Omega - (2 \eta_r(\phi_k) \bo, \curl \bm{v})_\Omega \\
\notag & \quad + (\bu_{\tau}, \bm{v}_\tau)_{\Gamma} - (\mathcal{L}(\psi) \nS \psi_k, \bm{v}_\tau)_{\Gamma} - (\mu \nabla \phi_k, \bm{v})_\Omega  + (( \bm{J} \cdot \nabla) \bu, \bm{v})_\Omega \\
\notag & \quad +(\eps |\nabla \bu|^{q-2} \nabla \bu, \nabla \bm{v})_\Omega + (\eps |\bu|^{q-2} \bu, \bm{v})_\Omega\\
\notag & \quad  + \frac{1}{2}( [\div \bm{J} - \tfrac{1}{h}(\rho - \rho_k) - \bu \cdot \nabla \rho_k] \bu , \bm{v})_\Omega, \\[1ex]
\label{Time:w} 0 & = \frac{1}{h}(\rho \bo - \rho_k \bo_k, \bm{z})_\Omega + (\div (\rho_k \bu \otimes \bo), \bm{z})_\Omega + (\bo_\tau, \bm{z}_\tau)_\Gamma \\
\notag & \quad + (c_0(\phi_k) \div \bo, \div \bm{z})_\Omega + (2 c_d(\phi_k) \D \bo, \D \bm{z})_\Omega + (2 c_a(\phi_k) \W \bo, \W \bm{z})_\Omega \\
\notag & \quad - (2 \eta_r(\phi_k) (\curl \bu - 2 \bo), \bm{z})_\Omega + ( (\bm{J} \cdot \nabla) \bo, \bm{z})_\Omega \\
\notag & \quad +(\eps |\nabla \bo|^{q-2} \nabla \bo, \nabla \bm{z})_\Omega + (\eps |\bo|^{q-2} \bo, \bm{z})_\Omega \\
\notag & \quad  + \frac{1}{2}( [\div \bm{J} - \tfrac{1}{h}(\rho - \rho_k) - \bu \cdot \nabla \rho_k] \bo , \bm{z})_\Omega, \\[1ex]
\label{Time:phi} 0 & = \frac{1}{h}(\phi - \phi_k) + \bu \cdot \nabla \phi_k - \div (m(\phi_k) \nabla \mu) \quad \text{ a.e.~in } \Omega, \\[1ex]
\label{Time:mu} \mu & = - \Delta \phi + F_{0,\kappa}'(\phi) + F_1'(\phi_k) + \frac{\sigma}{h} (\phi - \phi_k) \quad \text{ a.e.~in } \Omega, \\[1ex]
\label{Time:psi} 0 & = \frac{1}{h}(\psi - \psi_k) + \bu_\tau \cdot \nS \psi_k + \mathcal{L}(\psi) \quad \text{ a.e.~on } \Gamma, \\[1ex]
\label{Time:L}  \mathcal{L}(\psi) & = \pdnu \phi + \zeta \psi + G_0'(\psi) + G_1'(\psi_k) - \Delta_\Gamma \psi \quad \text{ a.e.~on } \Gamma.
\end{alignat}
\end{subequations}

\begin{remark}\label{rem:R:J}
In \eqref{Time:u} we have used the identity 
\[
\div ( \bm{J} \otimes \bu) = (\div \bm{J}) \bu + (\bm{J} \cdot \nabla) \bu
\]
and the following discretisation of \eqref{PF:rho:equ:R}
\[
\frac{1}{h}(\rho - \rho_k) + \div (\rho_k \bu + \bm{J}) = R_{k+1}, \quad R_{k+1} = - m(\phi_k) \nabla \mu \cdot \nabla \rho'(\phi_k),
\]
where $\bm{J}$ is defined in \eqref{disc:J}, to obtain the relation
\begin{align*}
& (\div (-\bm{J} \otimes \bu), \bm{v})_\Omega + \frac{1}{2}(R_{k+1} \bu, \bm{v})_\Omega \\
& \quad =  -((\bm{J} \cdot \nabla) \bu, \bm{v})_\Omega + \frac{1}{2} ( [\tfrac{1}{h}(\rho - \rho_k) + \bu \cdot \nabla \rho_k - \div \bm{J}] \bu, \bm{v})_\Omega,
\end{align*}
which serves as a discretisation of $\div (-\bm{J} \otimes \bu) + \frac{1}{2} R \bu$ in \eqref{app:model:u}. A similar relation also appears in \eqref{Time:w} for the discretisation of $\div (- \bm{J} \otimes \bo) + \frac{1}{2} R \bo$ in \eqref{app:model:w}.
\end{remark}

\begin{remark}\label{rem:cc:split}
The time discretisations of the derivatives of the free energy densities $F_\kappa$ and $G$ in \eqref{Time:mu} and \eqref{Time:L}, respectively, follows from the convex-concave splitting approach \cite{EllStu,Eyre}. In particular, using the convexity of $F_{0,\kappa}$ and concavity of $F_1$ we have the inequality
\[
(F_{0,\kappa}'(\phi) + F_1'(\phi_k), \phi - \phi_k)_{\Omega} \geq \int_\Omega F_{0,\kappa}(\phi) + F_1(\phi) \, dx - \int_\Omega F_{0,\kappa}(\phi_k) + F_1(\phi_k) \, dx.
\]
\end{remark}
Introducing the notation
\[
\UU^\eps := \begin{cases} \HH^1 & \text{ if } \eps = 0, \\
\WW^{1,q} & \text{ if } \eps > 0,
\end{cases} \quad \text{ and } \quad \UU^\eps_{\div} := \begin{cases} \HH^1_{\div} & \text{ if } \eps = 0, \\
\WW^{1,q}_{\div} & \text{ if } \eps > 0,
\end{cases} 
\]
the existence of a time discrete solution to \eqref{Time:dis} can be formulated as follows:
\begin{lem}\label{lem:discrete}
Let $\bu_k \in \HH_{\div}$, $\bo_k \in \HH$, $(\phi_k, \psi_k) \in \VV^2$ and $\rho_k = \rho(\phi_k)$ be given, then under  \eqref{ass:dom}, \eqref{ass:coeff}, \eqref{ass:G}, \eqref{ass:rho:approx} and \eqref{ass:F:approx}, for any $\eps, \sigma \in [0,1]$, $\kappa \in (0,1]$, there exists $(\bu, \bo, \mu, (\phi, \psi))$ with the regularity
\[
\bu \in \UU^\eps_{\div}, \quad \bo \in \UU^\eps, \quad (\phi, \psi) \in \VV^2, \quad \mu \in H^2_n(\Omega),
\]
and solving \eqref{Time:dis}. In addition, the discrete energy inequality holds:
\begin{equation}\label{dis:energy:est}
\begin{aligned}
& E_\kappa(\bu, \bo, \phi, \psi) + \int_\Omega \frac{\rho_k}{2} (|\bu - \bu_k|^2 + |\bo - \bo_k|^2 ) \, dx +\int_\Omega \frac{1}{2} |\nabla (\phi - \phi_k) |^2\, dx  \\
&  \qquad + h \int_\Omega 2 \eta(\phi_k) |\D \bu|^2 + 4 \eta_r(\phi_k)  \big | \frac{1}{2} \curl \bu - \bo \big |^2 + \eps (|\nabla \bu|^q + |\bu|^q) \, dx \\
&  \qquad + h \int_\Omega c_0(\phi_k) |\div \bo|^2 + 2 c_d(\phi_k) |\D \bo|^2 + 2 c_a(\phi_k) |\W \bo|^2 \, dx \\
& \qquad + h \int_\Omega \eps (|\nabla \bo|^q + |\bo|^q) \, dx + h \int_\Gamma |\bu_\tau|^2 + |\bo_\tau|^2 + |\mathcal{L}(\psi)|^2 \, dS \\
& \qquad + \frac{\sigma}{h^2} \int_\Omega | \phi - \phi_k |^2 \, dx + \frac{1}{2} \int_\Gamma \zeta | \psi - \psi_k |^2 + |\nS (\psi - \psi_k)|^2 \, dS \\
& \quad \leq E_\kappa(\bu_k, \bo_k, \phi_k, \psi_k).
\end{aligned}
\end{equation}
\end{lem}
Note that the solvability of \eqref{Time:dis} also holds for the limiting case $\sigma = \eps = 0$.

\begin{proof}
In order to utilize the Leray--Schauder principle to deduce the existence of a solution to the discrete problem \eqref{Time:dis} we define function spaces
\[
X := (\UU^\eps_{\div} \times \UU^\eps) \times H^2_n(\Omega) \times \VV^2 , \quad Y := ((\UU^\eps_{\div})^* \times (\UU^\eps)^*) \times L^2(\Omega)  \times \VV^0,
\]
and nonlinear operators $\mathcal{M}_k, \mathcal{F}_k : X \to Y$ as
\begin{align*}
\mathcal{M}_k(\bm{p}) = \begin{pmatrix} L_{k}(\bu,\bo) \\[1ex]  - \div (m(\phi_k) \nabla \mu) + \int_\Omega \mu \, dx \\[1ex] A_W(\phi, \psi) \end{pmatrix}, \quad \mathcal{F}_k(\bm{p}) = \begin{pmatrix} 
S_{\Omega} + S_{\Gamma} \\[1ex]
- \tfrac{1}{h}(\phi - \phi_k) - \bu \cdot \nabla \phi_k + \int_\Omega \mu \, dx \\[1ex]
B_k
\end{pmatrix},
\end{align*}
for $\bm{p} = (\bu, \bo, \mu, (\phi, \psi)) \in X$, where
\begin{align*}
\inn{L_{k}(\bu, \bo)}{(\bm{v}, \bm{z})} & := \int_\Omega \eps |\nabla \bu |^{q-2} \nabla \bu : \nabla \bm{v} + \eps |\bu|^{q-2} \bu \cdot \bm{v} + 2 \eta(\phi_k) \D \bu : \D \bm{v} \, dx \\
&  \quad + \int_\Omega   2 \eta_r(\phi_k) \W \bu : \W \bm{v}  - 2 \eta_r(\phi_k) \bo \cdot \curl \bm{v} \, dx + \int_\Gamma \bu_\tau \cdot \bm{v}_\tau \, dS \\
& \quad + \int_\Omega \eps |\nabla \bo |^{q-2} \nabla \bo : \nabla \bm{z} + \eps |\bo|^{q-2} \bo \cdot \bm{z} +  c_0(\phi_k) \div \bo \div \bm{z} \, dx \\
& \quad  +\int_\Omega  2 c_d(\phi_k) \D \bo : \D \bm{z} + 2 c_a(\phi_k) \W \bo: \W \bm{z} \, dx + \int_\Gamma \bo_\tau \cdot \bm{z}_\tau \, dS, \\
& \quad + \int_\Omega 4 \eta_r(\phi_k) \bo \cdot \bm{z} - 2 \eta_r(\phi_k) \curl \bu \cdot \bm{z} \, dx \\
\inn{S_{\Omega}}{\bm{v}} & := \int_\Omega \Big ( - \frac{1}{h}(\rho \bu - \rho_k \bu_k) - \div (\rho_k \bu \otimes \bu) + \mu \nabla \phi_k - (\bm{J} \cdot \nabla) \bu \Big ) \cdot \bm{v} \, dx \\
& \quad - \int_\Omega \frac{1}{2} \Big ( \div \bm{J} - \frac{1}{h}(\rho - \rho_k) - \bu \cdot \nabla \rho_k \Big ) \bu \cdot \bm{v} \, dx, \\
& \quad + \int_\Omega \Big ( - \frac{1}{h}(\rho \bo - \rho_k \bo_k) - \div (\rho_k \bu \otimes \bo) - (\bm{J} \cdot \nabla) \bo \Big ) \cdot \bm{z} \, dx \\
&\quad  - \int_\Omega \frac{1}{2} \Big ( \div \bm{J} - \frac{1}{h}(\rho - \rho_k) - \bu \cdot \nabla \rho_k \Big ) \bo \cdot \bm{z} \, dx, \\
\inn{S_{\Gamma}}{\bm{v}} & := \int_\Gamma \mathcal{L}(\psi) \nS \psi_k \cdot \bm{v}_\tau \, dS,
\end{align*}
for $\bm{v} \in \UU^\eps_{\div}$ and $\bm{z} \in \UU^\eps$, and $A_W: \VV^2 \to \VV^0$ is the Wentzell Laplacian:
\begin{align*}
A_W(\phi, \psi) := \begin{pmatrix} - \Delta \phi \\[1ex] - \Delta_\Gamma \psi + \pdnu \phi + \zeta \psi 
\end{pmatrix},
\end{align*}
along with 
\begin{align*}
B = \begin{pmatrix} 
\mu - F_{0,\kappa}'(\phi) - F_1'(\phi_k) - \frac{\sigma}{h} (\phi - \phi_k) \\[1ex]
\mathcal{L}(\psi) - G_0'(\psi) - G_1'(\phi_k)
\end{pmatrix}.
\end{align*}
In particular, $\bm{p} = (\bu, \bo, \mu, (\phi, \psi))$ is a weak solution to the time discrete problem \eqref{Time:dis} if and only if $\mathcal{M}_k(\bm{p}) = \mathcal{F}_k(\bm{p})$. Then, on account of the relation
\begin{align*}
|\curl \bu |^2 & = \eps_{ijk} \pd_j u_k \eps_{ilm} \pd_l u_m = (\delta_{jl} \delta_{km} - \delta_{jm} \delta_{kl}) \pd_j u_k \pd_l u_m \\
& = \pd_j u_k \pd_j u_k - \pd_j u_k \pd_k u_j = (\nabla \bu -( \nabla \bu)^{\top}) : \nabla \bu = 2 \W \bu : \nabla \bu = 2 \W \bu : \W \bu,
\end{align*}
we first observe that
\begin{align*}
\inn{L_{k}(\bu, \bo)}{(\bu, \bo)} & = \int_\Omega \eps \big ( |\nabla \bu|^q + |\bu|^q + |\nabla \bo|^q + |\bo|^q \big ) + 2 \eta(\phi_k) |\D \bu|^2 \, dx \\
& \quad + \int_\Omega c_0(\phi_k) |\div \bo|^2 + 2 c_d(\phi_k) |\D \bo|^2 + 2 c_a(\phi_k) |\W \bo|^2 \, dx  \\
& \quad + \int_\Omega 4 \eta_r(\phi_k) \Big | \frac{1}{2} \curl \bu - \bo \Big|^2  \, dx + \int_\Gamma |\bu_\tau|^2 + |\bo_\tau|^2 \, dS \\
& \geq c \Big (\| \bu \|_{\HH^1}^2 + \| \bo \|_{\HH^1}^2 + \eps \| \bu \|_{W^{1,q}}^q + \eps \| \bo \|_{W^{1,q}}^{q} \Big ),
\end{align*}
and hence $L_{k}$ is coercive:
\[
\lim_{\|(\bu, \bo)\|_{\UU^\eps \times \UU^\eps} \to \infty} \frac{\inn{L_{k}(\bu,\bo)}{(\bu,\bo)}}{\|(\bu, \bo)\|_{\UU^\eps \times \UU^\eps}} = \infty.
\]
Furthermore, it holds that
\begin{align*}
\inn{L_{k}(\bu,\bo) - L_{k}(\bm{v}, \bm{z})}{(\bu - \bm{v}, \bo - \bm{z})} \geq 0
\end{align*}
for all $\bu, \bm{v} \in \UU^\eps_{\div}$ and $\bo, \bm{z} \in \UU^\eps$, with equality achieved if and only if $\bu = \bm{v}$ and $\bo = \bm{z}$. This shows $L_{k}$ is a strictly monotone operator. Then, with analogous arguments in \cite[Proof of Lemma 3.3]{GGW} we deduce that the operator $\mathcal{M}_k : X \to Y$ is invertible with a continuous inverse $\mathcal{M}_k^{-1} : Y \to X$.  By restricting to the product Banach space
\[
\tilde{Y} := (\HH^{\frac{3}{4}}_{\div})^{*} \times (\HH^{\frac{3}{4}})^* \times W^{1,\frac{3}{2}}(\Omega) \times \Big ( W^{\frac{1}{2},2}(\Omega) \times W^{\frac{1}{4},2}(\Gamma) \Big ),
\]
and employing the compact embeddings $\tilde{Y} \Subset Y$, the restriction $\mathcal{M}_k^{-1} : \tilde{Y} \subset Y \to X$ is a compact operator. Next, we show that $\mathcal{F}_k : X \to \tilde{Y}$ is continuous and maps bounded sets into bounded sets. Our current setting differs from that of \cite{GGW} where there are new contributions in $S_{\Omega}$ involving $\bo$, but as the structure of these new contributions are similar to those for $\bu$, we may argue similarly as in \cite{GGW} to deduce the continuity of $\mathcal{F}_k : X \to \tilde{Y}$.

We define the mapping 
\[
\mathcal{K}_k := \mathcal{F}_k \circ \mathcal{M}_k^{-1} : \tilde{Y} \to \tilde{Y},
\]
which is a compact operator due to the compactness of $\mathcal{M}_k^{-1}$ and the continuity of $\mathcal{F}_k$. Then, $\bm{p}$ is a solution to the discrete system \eqref{Time:dis}, i.e., $\mathcal{M}_k(\bm{p}) = \mathcal{F}_k(\bm{p})$, if and only if $\bm{f} = \mathcal{M}_k(\bm{p})$ is a fixed point of $\mathcal{K}_k$, i.e., $\bm{f} = \mathcal{K}_{k}(\bm{f})$. Then, to apply the Leray--Schauder principle we have to show there exists $R > 0$ such that if $\bm{f} \in \tilde{Y}$ and $0 \leq \lambda \leq 1$ satisfies $\bm{f} = \lambda \mathcal{K}_k(\bm{f})$, then $\| \bm{f} \|_{\tilde{Y}} \leq R$. In particular, $\bm{f} = \lambda \mathcal{K}_k(\bm{f})$ is equivalent to $\mathcal{M}_k(\bm{p}) = \lambda \mathcal{F}_k(\bm{p})$, which reads as
\begin{subequations}\label{LS:cal}
\begin{alignat}{2}
\label{LS:u} &(\eps |\nabla \bu|^{q-2} \nabla \bu, \nabla \bm{v})_\Omega + (\eps |\bu|^{q-2} \bu, \bm{v})_\Omega +  (2 \eta(\phi_k) \D \bu, \D \bm{v})_\Omega \\
\notag & \qquad  + (2 \eta_r(\phi_k) \W \bu, \W \bm{v})_\Omega  - (2 \eta_r(\phi_k) \bu, \curl \bm{v})_\Omega + (\bu_{\tau}, \bm{v}_\tau)_{\Gamma}  \\
\notag & \quad  = -\tfrac{\lambda}{h}(\rho \bu - \rho_k \bu_k, \bm{v})_\Omega -\lambda (\div (\rho_k \bu \otimes \bu), \bm{v})_\Omega  +\lambda(\mu \nabla \phi_k, \bm{v})_\Omega   \\
\notag & \qquad - \lambda (( \bm{J} \cdot \nabla) \bu, \bm{v})_\Omega  - \tfrac{\lambda}{2}( [\div \bm{J} - \tfrac{1}{h}(\rho - \rho_k) - \bu \cdot \nabla \rho_k] \bu , \bm{v})_\Omega \\
\notag & \qquad + \lambda (\mathcal{L}(\psi) \nS \psi_k, \bm{v}_\tau)_{\Gamma},  \\[1ex]
\label{LS:w} & (\eps |\nabla \bo|^{q-2} \nabla \bo, \nabla \bm{z})_\Omega + (\eps |\bo|^{q-2} \bo, \bm{z})_\Omega \\
\notag & \qquad + ( 2 \eta_r(\phi_k)(2 \bo - \curl \bu), \bm{z})_\Omega  + (\bo_\tau, \bm{z}_\tau)_\Gamma \\
\notag & \qquad + (c_0(\phi_k) \div \bo, \div \bm{z})_\Omega + (2 c_d(\phi_k) \D \bo, \D \bm{z})_\Omega + (2 c_a(\phi_k) \W \bo, \W \bm{z})_\Omega \\
\notag & \quad = -\tfrac{\lambda}{h}(\rho \bo - \rho_k \bo_k, \bm{z})_\Omega -\lambda (\div (\rho_k \bu \otimes \bo), \bm{z})_\Omega   \\
\notag & \qquad -\lambda ( (\bm{J} \cdot \nabla) \bo, \bm{z})_\Omega- \tfrac{\lambda}{2}( [\div \bm{J} - \tfrac{1}{h}(\rho - \rho_k) - \bu \cdot \nabla \rho_k] \bo , \bm{z})_\Omega, \\[1ex]
\label{LS:phi}  & - \div (m(\phi_k) \nabla \mu) + \int_\Omega \mu \, dx\\
\notag & \quad = -\tfrac{\lambda}{h}(\phi - \phi_k) -\lambda \bu \cdot \nabla \phi_k + \lambda \int_\Omega \mu \, dx \quad \text{ a.e.~in } \Omega, \\[1ex]
\label{LS:mu} & - \Delta \phi  = \lambda \mu  - \lambda F_{0,\kappa}'(\phi) -  \lambda F_1'(\phi_k) - \lambda \tfrac{\sigma}{h} (\phi - \phi_k) \quad \text{ a.e.~in } \Omega, \\[1ex]
\label{LS:psi} &  -\Delta_\Gamma \psi + \pdnu \phi + \zeta \psi  =  \lambda \mathcal{L}(\psi) - \lambda G_0'(\psi) - \lambda G_1'(\psi_k) \quad \text{ a.e.~on } \Gamma, 
\end{alignat}
\end{subequations}
along with (recalling \eqref{Time:psi})
\[
\mathcal{L}(\psi) = - \frac{1}{h}(\psi - \psi_k) - \bu_\tau \cdot \nS \psi_k \quad \text{ a.e.~on } \Gamma.
\]
We now derive uniform estimates in $\lambda \in [0,1]$ where the calculation also provides the discrete energy estimate \eqref{dis:energy:est}. Choosing $\bm{v} = \bu$ in \eqref{LS:u}, $\bm{z} = \bo$ in \eqref{LS:w}, testing \eqref{LS:phi} with $\mu$, \eqref{LS:mu} with $\frac{1}{h}(\phi - \phi_k)$, \eqref{LS:psi} with $\frac{1}{h}(\psi - \psi_k)$, as well as using the identities (see e.g. \cite[Lemma 4.3]{ADG})
\begin{align*}
( \tfrac{1}{2} (\div \bm{J}) \bm{q} +( \bm{J} \cdot \nabla) \bm{q}, \bu)_\Omega = \int_\Omega \tfrac{1}{2} \div ( |\bm{q}|^2 \bm{J}) \, dx = 0, \\
(\div (\rho_k \bu \otimes \bm{q}) - \tfrac{1}{2} (\bu \cdot \nabla \rho_k) \bm{q}, \bm{q})_\Omega = \int_\Omega  \tfrac{1}{2} \div (\rho_k |\bm{q}|^2 \bu) \, dx = 0,
\end{align*}
and
\begin{align*}
\frac{\rho \bm{q} - \rho_k \bm{q}_k}{h} \cdot \bm{q} = \frac{1}{2h} ( \rho |\bm{q}|^2 - \rho_k |\bm{q}_k|^2) + \frac{(\rho - \rho_k) |\bm{q}|^2}{2h} + \frac{\rho_k |\bm{q} - \bm{q}_k|^2}{2h},
\end{align*}
for $\bm{q} \in \{ \bu, \bo\}$, we then obtain upon summing
\begin{align}\label{LS:est}
\notag &  \int_\Omega \eps \Big ( |\nabla \bu|^q + |\bu|^q + |\nabla \bo|^q + |\bo|^q  \Big ) + 2 \eta(\phi_k) |\D \bu|^2 + 4 \eta_r(\phi_k) |\tfrac{1}{2} \curl \bu - \bo|^2 \, dx \\
\notag & \qquad + \int_\Omega c_0(\phi_k) |\div \bo|^2 + 2 c_d(\phi_k) |\D \bo|^2 + 2 c_a(\phi_k) |\W \bo|^2 + m(\phi_k) |\nabla \mu|^2  \, dx \\
\notag & \qquad + (1-\lambda) \Big | \int_\Omega \mu \, dx \Big |^2 + \int_\Gamma |\bu_\tau|^2 + |\bo_\tau|^2 \, dS + \lambda \int_\Omega \frac{\rho_k}{2h} \Big (|\bu - \bu_k|^2 +|\bo - \bo_k|^2 \Big )\, dx   \\
\notag & \qquad + \lambda \int_\Omega \frac{\rho |\bu|^2 - \rho_k |\bu_k|^2}{2h} + \rho_k \frac{|\bu - \bu_k|^2}{2h} + \frac{\rho |\bo|^2 - \rho_k |\bo_k|^2}{2h} + \rho_k \frac{|\bo - \bo_k|^2}{2h} \, dx \\ 
\notag & \qquad + \int_\Omega \frac{1}{2h} \big ( |\nabla \phi|^2 - |\nabla \phi_k|^2 + |\nabla (\phi - \phi_k)|^2 \big) + \frac{\sigma}{h^2} |\phi - \phi_k|^2 \, dx + \lambda \int_\Gamma |\mathcal{L}(\psi)|^2 \, dS \\
\notag & \qquad + \int_\Gamma \frac{1}{2h} \big ( |\nS \psi|^2 + \zeta |\psi|^2 - |\nS \psi_k|^2 - \zeta |\psi_k|^2 + |\nS(\phi - \psi_k)|^2 + \zeta |\psi - \psi_k|^2 \big ) \, dS \\
\notag & \quad = - \lambda \int_\Omega \tfrac{1}{h} (F_{0,\kappa}'(\phi) + F_1'(\phi_k) (\phi- \phi_k) \, dx - \lambda \int_\Gamma \tfrac{1}{h}(G_0'(\psi) + G_1'(\psi_k))(\psi - \psi_k) \, dS \\
 & \quad \leq - \lambda \int_\Omega F_\kappa(\phi) - F_\kappa(\phi_k) \, dx - \lambda \int_\Gamma G(\psi) - G(\psi_k) \, dS,
\end{align}
where we have used Remark \ref{rem:cc:split} for the inequality. Note that the discrete energy inequality \eqref{dis:energy:est} is obtained for $\lambda = 1$.

From \eqref{LS:est}, by neglecting several non-negative terms that do not contribute to the estimates of $\| \bm{p} \|_X$, we first infer the following estimate after multiplying across by $h$ and utilizing the lower bounds for $\eta$, $\eta_r$, $c_0$, $c_d$, $c_a$ and $m$ (cf.~\eqref{ass:coeff}), as well as for $F_\kappa$ and $G$ (cf.~\eqref{ass:G} and \eqref{ass:F:approx}):
\begin{equation}\label{LS:est:2}
\begin{aligned}
& h \int_\Omega  k_0 \big (2 |\D \bu|^2 + |\nabla \mu|^2 + |\div \bo|^2 +2 |\D \bo|^2  + 2 |\W \bo|^2 \big ) \, dx \\
& \qquad  + h \int_\Omega \eps \big ( |\nabla \bu|^q + |\bu|^q + |\nabla \bo|^q + |\bo|^q \big ) \, dx + (1-\lambda) h \Big | \int_\Omega \mu \, dx \Big |^2  \\
& \qquad + \frac{\lambda \sigma}{h} \int_\Omega |\phi - \phi_k|^2 \, dx  + h \int_\Gamma |\bu_\tau|^2 + |\bo_\tau|^2 + \lambda |\mathcal{L}(\psi)|^2 \, dS \\
& \qquad + \int_\Gamma \frac{1}{2} |\nS \psi|^2 + \frac{\zeta}{2} |\psi|^2  \, dS + \int_\Omega \frac{1}{2} |\nabla \phi|^2 + \frac{\delta}{2} |\phi|^2 \, dx \\
& \quad \leq  \int_\Omega \frac{1}{2} |\nabla \phi_k|^2 + F_\kappa(\phi_k) + \frac{\rho_k}{2} (|\bu_k|^2 + |\bo_k|^2 )\, dx +\int_\Gamma  \frac{1}{2} |\nS \psi_k|^2 + G(\psi_k) \, dS \\
& \qquad + c_G |\Gamma| + c_F |\Omega| \\
& \quad  \leq C_k
\end{aligned}
\end{equation}
where for the right-hand side we used the fact that $\lambda \leq 1$ and the constant $C_k$ depends on $h$ but is independent of $\lambda \in [0,1]$. This implies that
\begin{equation}\label{LS:est:3}
\begin{aligned}
& \| \bu \|_{\UU^\eps} + \| \bo \|_{\UU^\eps} + \sqrt{1-\lambda} \; \Big | \int_\Omega \mu \, dx \Big| \\
& \quad + \sqrt{\lambda} \| \mathcal{L}(\psi) \|_{L^2(\Gamma)} + \| \nabla \mu \| + \| (\phi, \psi) \|_{\VV^1} \leq C_k.
\end{aligned}
\end{equation}
It remains to show that the $H^2$-norm of $\mu$ and the $\VV^2$-norm of $(\phi, \psi)$ can be bounded independently of $\lambda$. Firstly, the $L^2$-norm of $\mu$ can be obtained in the case $\lambda \in [0,\frac{1}{2})$ from $|\int_\Omega \mu \, dx | \leq C_k$ and the Poincar\'e inequality. On the other hand, for $\lambda \in [\frac{1}{2},1]$, we integrate \eqref{LS:mu} over $\Omega$ and use \eqref{LS:psi} to deduce
\begin{align*}
\frac{1}{2} \int_\Omega \mu \, dx \leq \lambda \int_\Omega \mu \, dx & = \lambda \int_\Omega F_{0,\kappa}'(\phi) + F_1'(\phi_k) + \frac{\sigma}{h} (\phi - \phi_k) \, dx \\
& \quad+ \int_\Gamma \zeta \psi - \lambda \mathcal{L}(\psi) + \lambda G_0'(\psi) + \lambda G_1'(\psi_k) \, dS.
\end{align*}
Hence, on account of the growth assumptions for $G$ \eqref{ass:G} and for $F_\kappa$ \eqref{ass:F:approx}, as well as the uniform $\VV^1$-estimate for $(\phi, \psi)$ and the $L^2(\Gamma)$-estimate for $\sqrt{\lambda} \mathcal{L}(\psi)$, we find that
\begin{align*}
\frac{1}{2} \Big | \int_\Omega \mu \, dx \Big | \leq C \Big ( 1 + \| (\phi, \psi) \|_{\VV^1}^{p+1} + \sqrt{\lambda} \| \mathcal{L}(\psi) \|_{L^2(\Gamma)} \Big ) \leq C_k.
\end{align*}
This implies that for any $\lambda \in [0,1]$ we have 
\begin{align}\label{LS:est:4}
\| \bu \|_{\UU^\eps} + \| \bo \|_{\UU^\eps} + \| \mu \|_{H^1} + \| (\phi, \psi) \|_{\VV^1} \leq C_k.
\end{align}
Next, viewing \eqref{LS:mu} as an elliptic equation for $\mu$ with right-hand side $\alpha := - \frac{\lambda}{h}(\phi - \phi_k) - \lambda \bu \cdot \nabla \phi_k + \lambda \int_\Omega \mu \, dx$. From \eqref{LS:est:4} we see that $\alpha$ is uniformly bounded in $L^2(\Omega)$:
\[
\| \alpha \| \leq C_k \big (1 + \| \phi_k \| +  \| \phi \|  + \|\nabla \phi_k \|_{L^4} \| \bu \|_{L^4} \big ) \leq C_k,
\]
and hence with the following elliptic regularity estimate (see \cite{ADG})
\[
\| \mu \|_{H^2} \leq C \Big ( \| \mu \|_{H^1} + \| \alpha \| \Big ),
\]
we deduce that for all $\lambda \in [0,1]$, 
\begin{align}\label{LS:est:5}
\| \bu \|_{\UU^\eps} + \| \bo \|_{\UU^\eps} + \| \mu \|_{H^2} + \| (\phi, \psi) \|_{\VV^1} \leq C_k.
\end{align}
Lastly, we express \eqref{LS:phi} and \eqref{LS:psi} as an elliptic equation for $(\phi, \psi)$ involving the Wentzell Laplacian:
\[
\begin{cases}
- \Delta \phi = h_1 := \lambda (\mu - F_{0,\kappa}'(\phi) - F_1'(\phi_k) - \tfrac{\sigma}{h}(\phi - \phi_k)) & \text{ in } \Omega, \\
- \Delta_\Gamma \psi + \pdnu \phi + \zeta \psi = h_2 := \lambda \mathcal{L}(\psi) - \lambda G_0'(\psi) - \lambda G_1'(\psi_k) & \text{ on } \Gamma,
\end{cases}
\]
then invoking the regularity estimate (see e.g.~\cite[Lemma A.1]{MiranZelik} or \cite[Lemma B.4]{GGW})
\[
\| (\phi, \psi) \|_{\VV^2} \leq C \Big ( \| h_1 \| + \| h_2 \|_{L^2(\Gamma)} \Big ),
\]
we infer that 
\[
\| \bm{p} \|_{X} = \| \bu \|_{\UU^\eps} + \| \bo \|_{\UU^\eps} + \| \mu \|_{H^2} + \| (\phi, \psi) \|_{\VV^2} \leq C_k. 
\]
Using that $\mathcal{F}_k :X \to \tilde{Y}$ maps bounded sets into bounded sets and $\bm{f} = \mathcal{M}_k(\bm{p})$, we see that if $\bm{f} = \lambda \mathcal{K}_k(\bm{f}) = \lambda \mathcal{F}_k(\bm{p})$ then,
\[
\| \bm{f} \|_{\tilde{Y}} = \| \lambda \mathcal{F}_k(\bm{p}) \|_{\tilde{Y}} \leq C_k \big ( \| \bm{p} \|_{X} + 1 \big ) \leq C_k.
\]
Hence, we can apply the Leray--Schauder principle to deduce the existence of a solution $\bm{p}$ to the time discrete system \eqref{Time:dis}.
\end{proof}

\subsection{Proof of Theorem \ref{thm:appmodel}}
In this section we will pass to the limit $N\to \infty$ (equivalently $h \to 0$) to prove Theorem \ref{thm:appmodel}. We begin with initial conditions $(\bu_0, \bo_0, (\phi_0^N, \psi_0^N))$ where $\bu_0 \in \HH_{\div}$ and $\bo_0 \in \HH$, while $(\phi_0^N, \psi_0^N) \in \VV^2$ is a set of regularized initial condition constructed as follows (cf.~\cite[Lemma B.5]{GGW}): we solve the following linear parabolic problem
\begin{align*}
\begin{cases}
\pd_t u - \Delta u = 0 & \text{ in } Q, \\
\pd_t v - \Delta_\Gamma v + \pdnu u + \zeta v = 0 & \text{ on } \Sigma, \\
(u,v) \vert_{t = 0} = (\phi_0, \psi_0) & \text{ in } \Omega \times \Gamma,
\end{cases}
\end{align*}
and set $(\phi_0^N, \psi_0^N) := (u(t), v(t)) \vert_{t = \frac{1}{N}}$. Then, it holds that $(\phi_0^N, \psi_0^N) \in \VV^2$ and $(\phi_0^N, \psi_0^N) \to (\phi_0, \psi_0)$ strongly in $\VV^1$ as $N \to \infty$. Then, from the growth assumption of $F_\kappa$ and $G$, by the generalized dominated convergence theorem we have
\begin{align}\label{initial:FG}
F_\kappa(\phi_0^N) \to F_\kappa(\phi_0) \text{ strongly in } L^1(\Omega), \quad G(\psi_0^N) \to G(\psi_0) \text{ strongly in } L^1(\Gamma).
\end{align}
For $f \in \{ \bu, \bo, \mu, (\phi, \psi)\}$, we now define the piecewise constant interpolation $f^N(t)$ on $[-h, \infty)$ through 
\[
f^N(t) = f_k \quad \text{ for } t \in [(k-1)h, kh),
\]
where $k \in \mathbb{N} \cup \{0\}$, and introduce the notation
\[
f_h := f(t - h) 
\]
so that
\[
f^N(t) = f_{k+1}, \quad f_h^N(t) = f_{k} \quad \text{ for } t \in [kh, (k+1)h).
\]
We define the piecewise linear interpolation $\tilde{f}^N$:
\[
\tilde{f}^N(t) = \frac{t - kh}{h} f_{k+1} + \frac{(k+1)h - t}{h} f_k \quad \text{ for } t \in [kh,(k+1)h),
\]
so that 
\[
\pd_t \tilde{f}^N = \frac{f_{k+1} - f_k}{h} =: \pd_{t,h}^{-} f^N(t) \quad \text{ for } t \in [kh, (k+1)h).
\]
Based these definitions, we use the notations:
\begin{equation}\label{defn:JRN}
\begin{aligned}
\rho^N & := \rho(\phi^N), \quad \bm{J}^N := - \rho'(\phi_h^N) m(\phi_h^N) \nabla \mu^N, \\
R^N & := \pd_{t,h}^{-} \rho^N + \div (\rho_h^N \bu^N + \bm{J}^N).
\end{aligned}
\end{equation}
Then, for arbitrary $\bm{v} \in C^\infty_0(0,T;C^\infty_{\div}(\overline{\Omega}))$ we choose $\tilde{\bm{v}} := \int_{kh}^{(k+1)h} \bm{v} \, dt$ as a test function in \eqref{Time:u}, and likewise for arbitrary $\bm{z} \in C^\infty_0(0,T;C^\infty_*(\overline{\Omega}))$ we choose $\tilde{\bm{z}} := \int_{kh}^{(k+1)h} \bm{z} \, dt$ as a test function in \eqref{Time:w}, and then summing over $k \in \mathbb{N}$ to obtain
\begin{subequations}\label{Limit:u:w}
\begin{alignat}{2}
\label{Time:u:2}
0 & = \int_Q -\rho^N \bu^N \cdot \pd_{t,h}^{+} \bm{v} + \div (\rho_h^N \bu^N \otimes \bu^N) \cdot \bm{v} + 2 \eta(\phi_h^N) \D \bu^N : \D \bm{v} \, dx \, dt \\
\notag & \quad + \int_Q 2 \eta_r(\phi_h^N) \W \bu^N : \W \bm{v} - 2 \eta_r(\phi_h^N) \bo^N \cdot \curl \bm{v} - \mu^N \nabla \phi_h^N \cdot \bm{v}  \, dx \, dt \\
\notag & \quad + \int_Q \eps |\nabla \bu^N|^{q-2} \nabla \bu^N : \nabla \bm{v} + \eps |\bu^N|^{q-2} \bu^N \cdot \bm{v} \, dx \, dt + \int_{\Sigma} \bu^N_\tau \cdot \bm{v}_\tau \, dS \, dt \\
\notag & \quad - \int_Q  (\bm{J}^N \otimes \bu^N) \cdot \nabla \bm{v} + \frac{1}{2} R^N \bu^N \cdot \bm{v} \, dx \, dt  - \mathcal{L}(\psi^N) \nS \psi^N_h \cdot \bm{v}_\tau \, dS \, dt \\
\label{Time:w:2} 0 & = \int_Q -\rho^N \bo^N \cdot \pd_{t,h}^{+} \bm{z} + \div (\rho_h^N \bu^N \otimes \bo^N) \cdot \bm{z} \, dx \, dt + \int_{\Sigma} \bo^N_\tau \cdot \bm{z}_\tau \, dS \, dt \\
\notag & \quad + \int_Q c_0(\phi_h^N) \div \bo^N \cdot \div \bm{z} +2 c_d(\phi_h^N) \D \bo^N : \D \bm{z} + 2 \eta_r(\phi_h^N) \W \bo^N : \W \bm{z} \, dx \, dt \\
\notag & \quad - \int_Q 2 \eta_r (\phi_h^N) (\curl \bu^N - 2 \bo^N) \cdot \bm{z} + (\bm{J}^N \otimes \bo^N) \cdot \nabla \bm{z} + \frac{1}{2} R^N \bo^N \cdot \bm{z} \, dx \, dt \\
\notag & \quad + \int_Q \eps |\nabla \bo^N|^{q-2} \nabla \bo^N : \nabla \bm{z} + \eps |\bo^N|^{q-2} \bo^N \cdot \bm{z} \, dx \, dt,
\end{alignat}
\end{subequations}
where we have used Remark \ref{rem:R:J}, as well as applied a relation akin to an integration by parts on the time derivative:
\begin{align*}
\int_Q \pd_{t,h}^{-} (\rho^N \bu^N) \cdot \bm{v} \, dx \, dt & =  \int_Q (\rho^N\bu^N)(t)  \cdot \frac{(\bm{v}(t) - \bm{v}(t+h))}{h} \, dx \, dt  \\
& =: \int_Q \rho^N \bu^N \cdot \pd_{t,h}^+ \bm{v} \, dx \, dt.
\end{align*}
Analogously, from \eqref{Time:phi}-\eqref{Time:L} we deduce that 
\begin{subequations}\label{Limit:CH}
\begin{alignat}{2}
\label{Time:phi:2} 0 & = \int_Q - \phi^N \pd_{t,h}^+ \xi - \bu^N \phi_h^N \cdot \nabla \xi + m(\phi_h^N) \nabla \mu^N \cdot \nabla \xi \, dx \, dt, \\
\label{Time:psi:2} 0 & = \int_\Sigma -\psi^N \pd_{t,h}^+ \Upsilon + (\bu_\tau^N \cdot \nS \psi_h^N ) \Upsilon + \mathcal{L}(\psi^N) \Upsilon \, dS \, dt, \\
\label{Time:mu:2} \mu^N & = - \Delta \phi^N + F_{0,\kappa}'(\phi^N) + F_1'(\phi_h^N) + \sigma \pd_{t,h}^- \phi^N \quad \text{ a.e.~in } Q, \\
\label{Time:L:2} \mathcal{L}(\psi^N) & = - \Delta_\Gamma \psi^N + \pdnu \phi^N + \zeta \psi^N + G_0'(\psi^N) + G_1'(\psi_h^N) \quad \text{ a.e.~on } \Sigma,
\end{alignat}
\end{subequations} for arbitrary $\xi \in C^\infty_0(0,T;H^1(\Omega))$ and $\Upsilon \in C^\infty_0(0,T;L^2(\Gamma))$. We also define a piecewise linear interpolation of the energy:
\[
\mathcal{E}^N(t) := \frac{(k+1)h - t}{h} E_\kappa(\bu_k, \bo_k, \phi_k, \psi_k) + \frac{t - kh}{h} E_\kappa(\bu_{k+1}, \bo_{k+1}, \phi_{k+1}, \psi_{k+1})
\]
for $t \in [kh, (k+1)h)$, and a piecewise linear interpolation dissipation functional
\begin{align*}
\mathcal{D}^N(t) & := \int_\Omega 2 \eta(\phi_k) |\D \bu_{k+1}|^2 + 4 \eta_r(\phi_k) \Big | \frac{1}{2} \curl \bu_{k+1} - \bo_{k+1} \Big |^2  + m(\phi_k) |\nabla \mu_{k+1}|^2\, dx \\
& \quad + \int_\Omega c_0(\phi_k) |\div \bo_{k+1}|^2 + 2 c_d(\phi_k) | \D \bo_{k+1}|^2 + 2 c_a(\phi_k) |\W \bo_{k+1}|^2 \, dx \\
& \quad + \int_\Omega \eps \big ( |\nabla \bu_{k+1}|^{q} + |\bu_{k+1}|^q + |\nabla \bo_{k+1}|^q + |\bo_{k+1}|^q \big ) + \frac{\sigma}{h^2} |\phi_{k+1} - \phi_k|^2 \, dx\\
& \quad + \int_\Gamma |(\bu_{k+1})_{\tau}|^2 +  |(\bo_{k+1})_{\tau}|^2  + |\mathcal{L}(\psi_{k+1})|^2 \, dS 
\end{align*}
for $t \in (kh, (k+1)h)$, $k \in \mathbb{N} \cup \{0\}$.  Then, the discrete energy inequality \eqref{dis:energy:est} in Lemma \ref{lem:discrete} leads to 
\begin{align}\label{energy:ineq:bdd}
- \frac{d}{dt} \mathcal{E}^N(t) = \frac{E_\eps(\bu_{k}, \bo_k, \phi_k, \psi_k) - E_\eps(\bu_{k+1}, \bo_{k+1}, \phi_{k+1}, \psi_{k+1})}{h} \geq \mathcal{D}^N(t)
\end{align}
for all $t \in (kh, (k+1)h)$, $k \in \mathbb{N} \cup \{0\}$. Integrating \eqref{energy:ineq:bdd} with respect to time yields
\begin{align}\label{energy:ineq:bdd:int}
E_\kappa(\bu^N, \bo^N, \phi^N, \psi^N)(t) + \int_{s}^t \mathcal{D}^N(r) \, dr \leq E_\kappa(\bu^N, \bo^N, \phi^N, \psi^N)(s)
\end{align} 
for all $0 \leq s \leq t < T$ with $s,t \in h( \mathbb{N} \cup \{0\})$. 

Recalling \eqref{initial:FG}, it holds that $E_\kappa(\bu^N, \bo^N, \phi^N, \psi^N)(0)$ can be bounded independent of $N$, and subsequently by analogous arguments in \cite[Section 3.2]{GGW} or \cite[Section 5]{ADG}, we obtain the following compactness assertions:

\begin{lem}\label{lem:compact:N}
Along a non-relabelled subsequence $N \to \infty$, there exist limit functions $(\bu, \bo, \mu, (\phi, \psi))$ satisfying
\begin{itemize}
\item the initial conditions: $\bu(0) = \bu_0$, $\bo(0) = \bo_0$, $(\phi,\psi)(0) = (\phi_0, \psi_0)$,
\item for the order parameters and chemical potential
\begin{align*}
(\phi^N, \psi^N) & \to (\phi, \psi) && \text{ weakly* in } L^\infty(0,T;\VV^1) \cap L^2(0,T;\VV^2), \\
\sigma^{1/2} \pd_{t,h}^- \phi^N & \to \sigma^{1/2} \pd_t \phi && \text{ weakly in } L^2(Q), \\
\mu^N & \to \mu && \text{ weakly in } L^2(0,T;H^1(\Omega)), \\
(F_{0,\kappa}'(\phi^N) , F_1'(\phi_h^N))& \to (F_{0,\kappa}'(\phi) , F_1'(\phi)) && \text{ weakly in } L^2(Q)^2, \\
(G_0'(\psi^N), G_1'(\psi_h^N)) & \to (G_0'(\psi), G_1'(\psi)) && \text{ weakly in } L^2(\Sigma)^2, \\
\mathcal{L}(\psi^N) & \to \mathcal{L}(\psi) && \text{ weakly in } L^2(\Sigma), \\
\bm{J}^N & \to \bm{J} = - \rho'(\phi) m(\phi) \nabla \mu && \text{ weakly in } L^2(0,T;L^2(\Omega;\R^d)), 
\end{align*}
\item for $\bm{y} \in \{\bu, \bo\}$
\begin{align*}
\bm{y}^N & \to \bm{y} && \text{ weakly* in } L^\infty(0,T;\HH) \cap L^2(0,T;\HH^1), \\
\bu_i^N \bm{y}_j^N & \to \bu_i \bm{y}_j && \text{ strongly in } L^k(0,T;L^2(\Omega)) \text{ for } k \in [1,4/3), \, 1 \leq i,j \leq d, \\
\eps^{1/q} \bm{y}^N & \to \eps^{1/q} \bm{y} && \text{ weakly in } L^q(0,T;\WW^{1,q}) \text{ for } q > 2d, \\
\rho_h^N \bm{y}^N & \to \rho(\phi) \bm{y} && \text{ strongly in } L^2(0,T;\HH), 
\end{align*}
and for arbitrary $\bm{s} \in L^{\infty}(0,T;\WW^{1,4}(\Omega) \cap \HH^{\frac{3}{2}+\delta})$, $\delta > 0$ and $q \geq 12$,
\[
\int_0^T (R^N \bm{y}^N, \bm{s})_\Omega \, dt \to \int_0^T (R \bm{y}, \bm{s})_\Omega \, dt 
\]
with $R = - m(\phi) \nabla \mu \cdot \nabla \rho'(\phi)$,
\end{itemize}
and the quintuple of limit functions $(\bu, \bo, \mu, (\phi, \psi))$ that is a weak solution to \eqref{app:model1}-\eqref{app:model1:bc} in the sense of Definition \ref{defn:app:solution}.
\end{lem}

\begin{proof}
The compactness assertions for $\bu$, $\mu$, $(\phi, \psi)$, $\rho$, $\bm{J}$, $F_\kappa$, $G$ and $\mathcal{L}(\psi)$ can be obtained by the same arguments in \cite{GGW}, while the assertions for $\bo$ follows analogously as for $\bu$. Thus, we omit them and sketch the details for the weak convergence of $R^N \bm{y}^N$, for $\bm{y} \in \{\bu, \bo\}$, which is not explicitly stated in \cite{GGW}. To start, we use the definition \eqref{defn:JRN} of $R^N$ to deduce that for $\bm{y} \in \{\bu, \bo\}$:
\begin{align}\label{RNyN:weak}
(R^N \bm{y}^N, \bm{s})_\Omega = (\pd_{t,h}^- \rho^N, \bm{y}^N \cdot \bm{s})_\Omega - (\rho_h^N \bu^N + \bm{J}^N, \nabla (\bm{y}^N \cdot \bm{s}))_\Omega
\end{align}
for any smooth test function $\bm{s} \in \C^\infty_*(\overline{\Omega})$. Lipschitz continuity of $\rho$ and boundedness of $\pd_{t,h}^- \phi^N$ in $L^2(Q)$ implies $\pd_{t,h}^- \rho^N$ is bounded in $L^2(Q)$. Together with the boundedness of $\rho_h^N \bm{y}^N + \bm{J}^N$ in $L^2(Q)$, as well as the fact that $\bm{y}^N \cdot \bm{s}$ is bounded in $\HH^1$ for any $\bm{s} \in \HH^{3/2+\delta}$ with $\delta > 0$, we obtain
\begin{align*}
\Big |\int_Q R^N \bm{y}^N \cdot \bm{s} \, dx \, dt \Big | \leq C \| \bm{s} \|_{L^\infty(0,T;\HH^{\frac{3}{2} + \delta})}.
\end{align*}
Choosing $\delta \geq \frac{1}{4}$ yields the embedding $\bm{s} \in L^\infty(0,T;\HH^{3/2+\delta}) \subset L^8(0,T;\WW^{1,4})$, which means we may consider in \eqref{RNyN:weak} a test function $\bm{s} \in L^8(0,T;\WW^{1,4})$. 

Next, in addition to the weak convergences stated in Lemma \ref{lem:compact:N} we collect the strong convergences in \cite{GGW} for $\bm{y}^N$ deduced via interpolation
\begin{align*}
\bm{y}^N \to \bm{y} & \text{ strongly in } L^p(0,T;L^{4}(\Omega;\R^d)) \text{ for any } p \in [1,8/3), \\
\nabla \bm{y}^N \to \nabla \bm{y} & \text{ strongly in } L^2(0,T;L^4(\Omega;\R^{d \times d})),
\end{align*}
so that when combined with the weak convergences of $\bm{J}^N$ in $L^2(Q)$ and of $\pd_{t,h}^- \rho^N$ in $L^2(Q)$ we find that for arbitrary $\bm{s} \in L^8(0,T;\WW^{1,4} \cap \HH^{\frac{3}{2}+\delta})$, $\delta > 0$:
\begin{align*}
\int_0^T (\pd_{t,h}^- \rho^N - \rho'(\phi) \pd_t \phi , \bm{y} \cdot \bm{s})_\Omega + (\rho^N_h \bu^N - \rho(\phi) \bu + \bm{J}^N - \bm{J}, \nabla (\bm{y} \cdot \bm{s}))_\Omega \, dt \to 0.
\end{align*}
In the above we have used that $\bm{y} \cdot \bm{s} \in L^2(0,T;H^1(\Omega))$, and
\begin{align*}
& \Big | \int_0^T (\pd_{t,h}^- \rho^N , (\bm{y}^N - \bm{y}) \cdot \bm{s})_\Omega + (\rho_h^N \bu^N + \bm{J}^N, \nabla ( (\bm{y}^N - \bm{y}) \cdot \bm{s}))_\Omega \, dt  \Big | \\
& \quad \leq \|\pd_{t,h}^- \rho^N \|_{L^2(Q)} \| \bm{y}^N - \bm{y} \|_{L^2(0,T;L^4)} \| \bm{s} \|_{L^{\infty}(0,T;L^4)} \\
& \qquad + \big ( \| \rho_h^N \bu^N \|_{L^2(Q)} + \|\bm{J}^N \|_{L^2(Q)} \big ) \| \bm{y}^N - \bm{y} \|_{L^2(0,T;W^{1,4})} \| \bm{s} \|_{L^\infty(0,T;W^{1,4})} \\
& \quad  \to 0.
\end{align*}
Hence, it holds that 
\begin{align}\label{RNyN:lim}
\int_0^T (R^N \bm{y}^N, \bm{s})_\Omega \, dt \to \int_0^T (\rho'(\phi) \pd_t \phi, \bm{y} \cdot \bm{s})_\Omega - (\rho(\phi) \bu + \bm{J}, \nabla (\bm{y} \cdot \bm{s}))_\Omega \, dt.
\end{align}
Then, passing to the limit $N \to \infty$ in \eqref{Time:phi:2} yields
\[
0 = \int_Q - \phi \pd_t \xi - \phi \bu \cdot \nabla \xi + m(\phi) \nabla \mu \cdot \nabla \xi \, dx \, dt
\]
for arbitrary $\xi \in C^\infty_0(0,T;H^1(\Omega))$, which also yields the equivalent strong formulation
\[
0 = (\pd_t \phi, \tilde{\xi})_\Omega + (\bu \cdot \nabla \phi, \tilde{\xi})_\Omega + (m(\phi) \nabla \mu, \nabla \tilde{\xi})_\Omega 
\]
for arbitrary $\tilde{\xi} \in H^1(\Omega)$ and for a.e.~$t \in [0,T]$.  

Then, for $q \geq 12$ we find that $\tilde{\xi} = \rho'(\phi(t)) \bm{y}(t) \cdot \bm{s}(t) \in H^1(\Omega)$ is an admissible test function, so that when we compare the resulting identity with \eqref{RNyN:lim} we infer that
\[
\int_0^T (R^N \bm{y}^N, \bm{s})_\Omega \, dt \to - \int_Q (m(\phi) \nabla \mu \cdot \nabla \phi) \bm{y} \cdot \bm{s} \, dx \, dt = \int_0^T (R \bm{y} , \bm{s} )_\Omega \, dt.
\]
Establishing the energy inequality \eqref{energy:ineq:approx} can be done in a similar way as in \cite{GGW} by multiplying the discrete energy inequality \eqref{energy:ineq:bdd} with $\eta \in W^{1,1}(0,T)$ such that $\eta \geq 0$ and $\eta(T) = 0$, integrating by parts and applying the strong convergences of $\bu^N(t)$, $\bo^N(t)$ in $\HH$, of $(\phi^N(t), \psi^N(t))$ in $C^0(\overline{\Omega}) \times C^0(\Gamma)$ for a.e.~$t \in (0,T)$, and the lower semicontinuity of the norms.
\end{proof}

\subsection{Proof of Theorem \ref{thm:exist}}
With the establishment of Lemma \ref{lem:compact:N}, Step 1 of the proof of Theorem \ref{thm:exist} is complete. To continue, we first refine the assumption \eqref{ass:rho:approx} on the mass density function $\rho$ to further satisfy the requirement
\[
\rho(s) = \frac{\overline{\rho}_1 - \overline{\rho}_2}{2} s + \frac{\overline{\rho}_1 + \overline{\rho}_2}{2} \text{ for } s \in [-1,1].
\]
Thanks to \eqref{F0approx} and \eqref{ass:ini}, the right-hand side of the energy inequality \eqref{energy:ineq:approx} can be bounded uniformly in $\kappa,\sigma, \eps \in (0,1]$, then Step 2 studies the limiting behavior as $\kappa \to 0$ where we obtain the singular logarithmic potential \eqref{log} as well as the boundedness properties $|\phi| < 1$ a.e.~in $Q$ and $|\psi| \leq 1$ a.e.~on $\Sigma$. This shows that the mass density function $\rho(\phi)$ coincides with the affine linear function $\hat{\rho}(\phi)$ defined in \eqref{hat:rho:def}, and that $R = -m(\phi) \nabla \mu \cdot \nabla \rho'(\phi) = 0$. Lastly, in Step 3, from an analogous energy inequality to \eqref{energy:ineq:approx} (with $\kappa = 0$) and almost identical arguments in Step 2 allows us to pass to the limit $\sigma \to 0$ and $\eps \to 0$ to complete the proof.

\section{Obstacle potential and the deep quench limit}\label{sec:obstacle}
In this section we scale the convex singular part $F_0$ in \eqref{log} by a coefficient $\theta > 0$ (corresponding to the absolute temperature) and perform a deep quench limit $\theta \to 0$. In this limit the scaled logarithmic potential approaches a double obstacle potential:
\[
F_{\mathrm{obs}}(s) = \mathbb{I}_{[-1,1]}(s) + F_1(s), \quad \mathbb{I}_{[-1,1]}(s) =\begin{cases} 0 & \text{ if } s \in [-1,1], \\
+\infty & \text{ otherwise},
\end{cases}
\]
where $\mathbb{I}_{[-1,1]}$ is the indicator function of the set $[-1,1]$.  Due to the non-differentiability of $\mathbb{I}_{[-1,1]}$, the subdifferential of $F_{\mathrm{obs}}$ is used to replace $F'$ in \eqref{weak:mu}, which now reads as
\begin{align}\label{obs:mu}
\mu = - \Delta \phi + F_1'(\phi) + \xi, \quad \xi \in \pd \mathbb{I}_{[-1,1]}(\phi) = \begin{cases}
[0,\infty) & \text{ if } |\phi| = 1 \\
\{0\} & \text{ if } |\phi| < 1, \\
(-\infty,0] & \text{ if } |\phi| = -1.
\end{cases}
\end{align}
An alternative formulation of \eqref{obs:mu} as a variational inequality can be achieved by considering the convex set
\[
\mathcal{K} := \{ f \in H^1(\Omega) \, : \, |f| \leq 1 \text{ a.e.~in } \Omega \},
\]
and by recalling the following inequality for the subgradient $\xi \in \pd \mathbb{I}_{[-1,1]}(\phi)$:
\[
0 = \mathbb{I}_{[-1,1]}(\eta) - \mathbb{I}_{[-1,1]}(\phi) \geq \xi (\eta - \phi)
\]
for all $\eta, \phi$ satisfying $|\eta| \leq 1$ and $|\phi| \leq 1$. Hence, for arbitrary $(\eta, \zeta) \in \VV^1$ with $\eta \in \mathcal{K}$, we test \eqref{obs:mu} with $\eta - \phi$ and invoke \eqref{weak:L} after integrating by parts to see
\begin{equation}\label{obs:vi}
\begin{aligned}
0 & \leq (\nabla \phi, \nabla (\eta- \phi))_\Omega + (F_1'(\phi) - \mu, \eta - \phi)_\Omega \\
& \quad + (\nS \psi, \nS (\zeta - \psi))_\Gamma + (\zeta \psi + G'(\psi) - \mathcal{L}(\psi), \zeta - \psi)_\Gamma.
\end{aligned}
\end{equation}
Alternatively, using \eqref{weak:psi} we can recast \eqref{obs:vi} as
\begin{equation}\label{obs:vi:2}
\begin{aligned}
0 & \leq (\nabla \phi, \nabla (\eta- \phi))_\Omega + (F_1'(\phi) - \mu, \eta - \phi)_\Omega \\
& \quad + (\nS \psi, \nS (\zeta - \psi))_\Gamma + (\zeta \psi + G'(\psi) + \bu_\tau \cdot \nS \psi, \zeta - \psi) \\
& \quad + \inn{\pd_t \psi}{\zeta - \psi}_{H^1(\Gamma)}.
\end{aligned}
\end{equation}

\begin{defn}\label{defn:weaksoln:obs}
A quintuple of functions $(\bu, \bo, \bm{J}, (\phi, \psi))$ is a weak solution to \eqref{ana:bulk}-\eqref{ana:bc} with the double obstacle potential $F_{\mathrm{obs}}$ on the time interval $[0,T]$ if the following properties are satisfied:
\begin{itemize}
\item Regularity
\begin{equation*}
\begin{aligned}
\bu & \in C_w([0,T]; \HH_{\div}) \cap L^2(0,T;\HH^1), \\
 \bo & \in C_w([0,T]; \HH) \cap L^2(0,T;\HH^1), \\
(\phi, \psi) & \in C_w([0,T];\VV^1) \cap L^2(0,T;\VV^2), \\
 \mu & \in L^2(0,T;H^1(\Omega)),
\end{aligned}
\end{equation*}
along with $|\phi| \leq 1$ a.e.~in $Q$ and $|\psi| \leq 1$ a.e.~on $\Sigma$.
\item Equations \eqref{weak:ns}, \eqref{weak:w}, \eqref{weak:phi}, \eqref{weak:psi} holding for all $\bm{v} \in C^\infty_0(0,T;\C^\infty_{\div}(\overline{\Omega}))$, $\bm{z} \in C^\infty_0(0,T;\C^\infty_*(\overline{\Omega}))$, $\xi \in C^\infty_0(0,T;C^1(\overline{\Omega}))$, and $\Upsilon \in C^\infty_0(0,T;C^0(\Gamma))$, along with \eqref{weak:J}, \eqref{weak:rho}, \eqref{weak:L} and
\begin{equation}\label{weak:vi}
\begin{aligned}
0 & \leq (\nabla \phi, \nabla (\eta- \phi))_Q + (F_1'(\phi) - \mu, \eta - \phi)_Q \\
& \quad + (\nS \psi, \nS (\zeta - \psi))_\Sigma + (\zeta \psi + G'(\psi) - \mathcal{L}(\psi), \zeta - \psi)_\Sigma,
\end{aligned}
\end{equation}
holding for all $(\eta, \psi) \in L^2(0,T;\VV^1)$ with $\eta \in L^2(0,T;\mathcal{K})$.
\item Energy inequality \eqref{energy:ineq} holds for all $t \in [s,\infty)$ and almost all $s \in [0,\infty)$ including $s = 0$.
\item Initial conditions
\[
(\bu, \bo, \phi, \psi) \vert_{t = 0} = (\bu_0, \bo_0, \phi_0, \mathrm{tr}(\phi_0)).
\]
\end{itemize}
\end{defn}
We point out that the main difference with Definition \ref{defn:weaksoln} is that $\phi$ is allowed to attain the values $\pm 1$. The main result for this section is formulated as follows:
\begin{thm}\label{thm:exist:vi}
Under Assumption \ref{ass:main}, for any $T \in (0,\infty)$ there exists a weak solution to \eqref{ana:bulk}-\eqref{ana:bc} with double obstacle potential $F_{\mathrm{obs}}$ in the sense of Definition \ref{defn:weaksoln:obs}.
\end{thm}

\subsection{Approximation scheme and uniform boundedness}
For $\theta \in (0,1]$ we consider the scaled logarithmic potential
\[
F_\theta(s) = \theta ((1+s) \ln (1+s) + (1-s) \ln (1-s)) + F_1(s) =: \theta F_0(s) + F_1(s),
\]
and denote by $(\bu_\theta, \bo_\theta, \mu_\theta, (\phi_\theta, \psi_\theta))$ as the weak solution associated to \eqref{ana:bulk}-\eqref{ana:bc} and initial conditions $(\bu_0, \bo_0, (\phi_0, \psi_0))$ with $F_\theta$ replacing $F$ in the sense of Definition \ref{defn:weaksoln}. Since $F_0 \in C^0([-1,1]) \cap C^2((-1,1))$ and $|\phi_0| \leq 1$ a.e.~in $\Omega$ (see \eqref{ass:ini}), it holds that $\theta F_0(\phi_0) \in L^1(\Omega)$ uniformly in $\theta$.  Hence, by Theorem \ref{thm:exist}, the energy inequality \eqref{energy:ineq} is valid for $(\bu_\theta, \bo_\theta, \mu_\theta, (\phi_\theta, \psi_\theta))$ with initial energy $E(\bu_0, \bo_0, \phi_0, \psi_0)$ bounded uniformly in $\theta \in [0,1]$. Hence, we deduce the following uniform estimates:
\begin{align*}
\bu_\theta, \bo_\theta & \text{ bounded in } L^\infty(0,T;\HH) \cap L^2(0,T;\HH^1), \\
\nabla \mu_\theta & \text{ bounded in } L^2(0,T;L^2(\Omega;\R^d)), \\
\mathcal{L}(\psi_\theta) & \text{ bounded in } L^2(\Sigma), \\
(\phi_\theta, \psi_\theta) & \text{ bounded in } L^\infty(0,T;\VV^1), \\
\theta F_0(\phi_\theta) & \text{ bounded in } L^\infty(0,T;L^1(\Omega)), \\
G(\psi_{\theta}) & \text{ bounded in } L^\infty(0,T;L^1(\Gamma)).
\end{align*}
Next, from \eqref{weak:phi} we infer that
\[
\int_\Omega \phi_\theta(t) \, dx = \int_\Omega \phi_0 \, dx, \quad \pd_t \phi_\theta \text{ bounded in } L^2(0,T;H^1(\Omega)^*).
\]
Subsequently, arguing as in Step 1 of \cite[Section 7, Theorem 3.5]{GGM} entails that
\begin{align*}
\theta F_0'(\phi_\theta) & \text{ bounded in } L^2(0,T;L^1(\Omega)), \\
\mu_\theta & \text{ bounded in } L^2(0,T;H^1(\Omega)),
\end{align*}
while the arguments in Step 2 of \cite[Section 7, Theorem 3.5]{GGM} furnish us with
\begin{align*}
\sqrt{\theta} F_0'(\phi_\theta) & \text{ bounded in } L^2(0,T;L^2(\Omega)), \\
F_0(\psi_\theta) & \text{ bounded in } L^\infty(0,T;L^1(\Gamma)).
\end{align*}
In particular, regularity theory with the Wenztell Laplacian (\cite[Lemma A.1]{MiranZelik} or \cite[Lemma B.4]{GGW}) applied to \eqref{weak:mu} and \eqref{weak:L} provides the following
\[
(\phi_\theta, \psi_\theta) \text{ bounded in } L^2(0,T;\VV^2).
\]
Together with the boundedness of $\bu_\theta$ in $L^2(0,T;\HH^1)$, from \eqref{weak:psi} we deduce that $\bu_{\theta,\tau} \cdot \nS \psi_\theta$ is bounded in $L^{\frac{4}{3}}(0,T;L^2(\Gamma))$ in three spatial dimensions. This yields that
\[
\pd_t \psi_\theta \text{ bounded in } L^{\frac{4}{3}}(0,T;L^2(\Gamma)).
\]
\subsection{Passing to the limit}
The above uniform boundedness are sufficient to pass to the limit $\theta \to 0$ to obtain limit functions $(\bu, \bo, \mu, (\phi, \psi))$ satisfying equations \eqref{weak:ns}, \eqref{weak:w}, \eqref{weak:phi}, \eqref{weak:psi} holding for all $\bm{v} \in C^\infty_0(0,T;\C^\infty_{\div}(\overline{\Omega}))$, $\bm{z} \in C^\infty_0(0,T;\C^\infty_*(\overline{\Omega}))$, $\xi \in C^\infty_0(0,T;C^1(\overline{\Omega}))$, and $\Upsilon \in C^\infty_0(0,T;C^0(\Gamma))$, along with \eqref{weak:J} and \eqref{weak:rho}, \eqref{weak:L}. The argument is identical to Step 3 of the proof of Theorem \ref{thm:exist} and thus we omit the details.  Let us consider an arbitrary $(\eta, \zeta) \in L^2(0,T;\VV^1)$ with $\eta \in L^2(0,T;\mathcal{K})$ and test \eqref{weak:mu} with $\eta - \phi_\theta$ to obtain
\begin{align*}
& ( F_1'(\phi_\theta) - \mu, \eta - \phi_\theta)_Q + (\nabla \phi_\theta, \nabla (\eta - \phi_\theta))_Q \\
& \qquad + (\nS \psi_\theta, \nS (\zeta - \psi_\theta))_\Sigma + ( \zeta \psi_\theta + G'(\psi_\theta) - \mathcal{L}(\psi_\theta), \zeta - \psi_\theta)_\Sigma  \\
& \quad = - \theta(F_0'(\phi_\theta), \eta - \phi_\theta)_Q \geq \theta (F_0(\phi_\theta), 1)_Q - \theta (F_0(\eta), 1)_Q \geq - \theta (F_0(\eta), 1)_Q,
\end{align*}
where we have used the convexity and non-negativity of $F_0$.  Integrating over $(0,T)$ and passing to the limit $\theta \to 0$ yields \eqref{weak:vi}.  This completes the proof of Theorem \ref{thm:exist:vi}.

\section*{Acknowledgements}
\noindent KFL gratefully acknowledges the support by the Research Grants Council of the Hong Kong Special Administrative Region, China [Project No.: HKBU 14303420].

\end{document}